\documentclass[11pt, a4paper]{amsart} 
\usepackage[margin=3cm]{geometry}
\usepackage[utf8]{inputenc}
\usepackage{csquotes}
\usepackage[T1]{fontenc}
\usepackage{silence}
\usepackage[
bibstyle=ieee,
citestyle=numeric,
sorting=nyvt,
url=true,
bibencoding=utf8,
backend=biber
]{biblatex}
\usepackage[english]{babel} 
\usepackage{amsthm, amsmath, amsfonts,amssymb}
\usepackage{dsfont}
\usepackage{tikz}
\usepackage{hyperref}
\usepackage{pgfplots}
\usetikzlibrary{
    arrows.meta,calc,positioning,backgrounds
}
\usepackage{esint}
\hypersetup{
    colorlinks=true,
    citecolor=black,
    linkcolor=blue,
    filecolor=black,      
    urlcolor=blue,
    pdfauthor=Gabriele Cassese,
    pdfcreator=Gabriele Cassese,
    pdftitle=Korn's inequality from the viewpoint
    of calculus of variations
}

\newtheorem{theorem}{Theorem}[section]
\newtheorem{proposition}{Proposition}[section]
\newtheorem{corollary}{Corollary}[theorem] 
\newtheorem{lemma}[theorem]{Lemma}
\theoremstyle{definition}
\newtheorem{definition}{\textsl{Definition}}

\theoremstyle{remark}
\newtheorem{remark}{Remark}
\newtheorem*{conjecture}{Conjecture}
\newcommand{\R}{\mathbb{R}}
\newcommand{\C}{\mathbb{C}}
\newcommand{\N}{\mathbb{N}}

\newcommand*\diff{\mathop{}\!\mathrm{d}}

\newcommand{\vvvert}[1]{
    {\left\vert\kern-0.25ex\left\vert\kern-
    0.25ex\left\vert #1 
    \right\vert\kern-0.25ex\right\vert\kern-
    0.25ex\right\vert}
}
    
\newcommand{\mat}[1]{
    \begin{pmatrix} #1 \end{pmatrix}
}

\newcommand{\abs}[1]{|#1|}

\newcommand{\norm}[1]{\|#1\|}
\pgfplotsset{compat=1.18}
\title{
Korn's inequality from the viewpoint of
calculus of variations
}

\begin{document}


\author{
    Gabriele Cassese
}

\address{
University of Oxford,
Mathematical Institute,
Radcliffe Observatory Quarter,
OX2 6GG,
Oxford,
United Kingdom
}

\email{
    gabriele.cassese@maths.ox.ac.uk
}

\begin{abstract}
We study the best possible constants in Korn-type
inequalities and their connection with
Morrey's problem in the calculus of variations.
We adapt techniques from the analysis of the
Beurling-Ahlfors transform to Korn's inequality.
In dimension $2$, we obtain a bound that is sharp up to a factor of $1.158$. In general, we show that the constant
in Korn's inequality admits a dimension-free bound, 
and we obtain an estimate that is sharp
up to a factor of $\sqrt 3$. 
We also establish several improvements to estimates
in various other function spaces.
Using a weighted version of 
Burkholder's differential subordination theorem, 
recently introduced in 
[J. Reine Angew. Math. 824 (2025), pp. 137-166],
we also prove a dimension-free
weighted version of the inequality
for Muckenhoupt weights.
\end{abstract}
\maketitle
\section{Introduction}
This paper addresses the problem dual to the one 
studied in~\cite{Cassese2025}:
obtaining upper bounds for the optimal constant
in Korn-type inequalities.
Korn's inequality, proved by
Korn~\cite{korn1908, korn1909} in $L^2$ and by 
Krylov~\cite{Krylov1962Dirichlet}
in the general case, states that
for an open subset $\Omega$ of
$\mathbb R^d$ and an exponent $p\in(1,\infty)$
there exists a constant $C(p,\Omega)$ such that
\begin{equation}\label{Korn}
    \left(
        \int_{\Omega} |\mathcal A(u)|^p
    \right)^{\frac 1p}
    \le
    C(p,\Omega)
    \left(
        \int_{\Omega}|\mathcal E(u)|^p
    \right)^{\frac 1p},
\end{equation}
holds for all Sobolev maps $u\in W^{1,p}_0(\Omega)$
where 
$\mathcal A(u)=P_{\mathrm{Skew}}(\nabla u)
\mathrel{:}=
1/2(\nabla u-\nabla u^t)$,
$\mathcal E(u)
\mathrel{:}=
1/2(\nabla u+\nabla u^t)$
and $|\cdot|$ is the Frobenius norm
on $d\times d$ matrices.
One often writes this equivalently as
\begin{equation}\label{eq:Kornfull}
    \left(
        \int_{\Omega}|\nabla u|^p
    \right)^{\frac1p}
    \le
    C_{\mathrm{KN}}(p,\Omega)
    \left(
        \int_{\Omega}|\mathcal E(u)|^p
    \right)^{\frac1p}.
\end{equation}
A stronger trace-free version of Korn's inequality 
also holds:
\begin{equation}\label{Korntrace}
    \left(
        \int_{\Omega} |\mathcal A_0(u)|^p
    \right)^{\frac 1p}
    \le
    C_0(p,\Omega,d)
    \left(
        \int_{\Omega}|\mathcal E_0(u)|^p
    \right)^{\frac 1p},
\end{equation}
where 
$\mathcal E_0(u)
=
\mathcal E(u)-
\mathrm{Tr}(\nabla u)(\mathrm{Id})/d$
and
$\mathcal A_0(u)
=
\mathcal A(u)+
\mathrm{Tr}(\nabla u)(\mathrm{Id})/d$.
These inequalities are particular cases
of the more general Calderón-Zygmund type inequality
\begin{equation}
    \|D^k(f)\|_{L^p}\lesssim_{p,P} \|P(f)\|_{L^p}
    \label{eq:genkorn}
\end{equation}
that is valid for any elliptic homogeneous 
differential operator of order $k$ with constant 
coefficients.
It is well known that the 
inequality~\eqref{eq:genkorn} fails
in the cases $p=1$ and $p=\infty$ unless $P$
is trivial
(see~\cite{Ornstein, deLeeuw, KirchheimJan}).
In the case of~\eqref{Korn},
the failure is well understood.
By~\cite{Breit2011}, Korn's inequality holds in the 
homogeneous
Sobolev-Orlicz space $\dot{W}^1L^\Phi$ 
if and only if
$\Phi$ satisfies both
the $\nabla_2$ and $\Delta_2$ conditions
(see section~\ref{sec:orlicz} for details).
This result clearly identifies the obstruction
at the endpoints:
at $p=1$, the function $\Phi_1(t)=t$
satisfies $\Delta_2$ but not $\nabla_2$
while for $p=\infty$
\[\Phi_\infty(t):=\begin{cases}
t,\ &t\le 1\\ 
\infty,\ &t>1
\end{cases}\] satisfies $\nabla_2$
but not $\Delta_2$.
Most of the literature on 
the inequality has focused on
the $L^2$ case. In that setting,
the inequality is a direct consequence
of the identity \[\|\mathcal 
E(u)\|_{L^2}^2=\|
\mathcal A(u)\|_{L^2}^2+
\|\text{div}(u)\|_{L^2}^2,
\] 
which implies
$C(2,\Omega)= 
1$. The constant is in fact
domain-independent for all $p\in(1,\infty)$:
$C(p,\Omega)=C(p,\R^d)$. Indeed,
approximation with compactly supported maps
shows that
\begin{equation*}
C(p,\R^d)=\lim_{r\to \infty} C(p, B_{\R^d}(0,r))
=C(p,B_{\R^d}(0,1))\le C(p, \R^d)
\end{equation*}
where the second equality follows from
the dilation invariance of Korn's inequality.
Since for any domain $\Omega$ containing
a ball $B(x_0,r)$ we have 
\begin{equation*}
C(p, B(x_0,r))\le C(p,\Omega)\le C(p,\R^d),
\end{equation*}
the result follows.
Accordingly, we henceforth write $C(p,d)$ to denote
Korn's constant for domains in $\R^d$.
Since a similar result holds for $C_0$
and $C_{\mathrm{KN}}$, we will drop the dependence
on $\Omega$ for them as well.\par
Korn's inequality has many applications,
notably in linearised elasticity, where it 
is employed to prove existence and regularity results
through coercivity estimates, and in
fluid mechanics; see~\cite{applied} and references
therein. Moreover, in recent years a 
new effort has been made to obtain simpler proofs
of the inequality (see for instance~\cite{PoinKorn}
and references therein) and to 
extend it to different function spaces
(see for instance~\cite{spector}).\par
The main objective of this paper is
to prove an almost sharp 
dimension-free bound for Korn's inequality.
More precisely, we prove the following:
In $2$ dimensions we obtain a substantially sharper
bound, which determines the best constant up to a
factor of $1.158$:
\begin{theorem}\label{thm:Korn2d}
    Given $u\in \dot{W}^{1,p}(\R^2, \R^2)$ we have
\begin{equation*}
    \|\mathcal A (u)\|_{L^p}\le 
    1.158(p^*-1)\|\mathcal E(u)\|_{L^p},
\end{equation*}
where $p^*=\max(p',p)$ and $p'$ is the
Hölder conjugate of $p$. Moreover, the best constant
satisfies
    \begin{equation}
        p^*-1\le C(p,2)\le 1.158(p^*-1).
    \end{equation}
\end{theorem}
\begin{remark}
    The conjectured equality $C(p,2)=p^*-1$ is closely
    related to the Iwaniec conjecture (see
    Section~\ref{sec:prereq} for more details).
    Both conjectures can be restated in terms of quasiconvexity at $0$
    of certain rank-one convex integrands, but in Iwaniec's case, where the focus lies on the decomposition of 
     $2\times 2$ matrices in
    their conformal and anticonformal parts, the associated
    integrand is isotropic. Isotropy has long been 
    considered one of the main properties to take advantage
    of when trying to prove quasiconvexity (see \cite{Guerra2022}), but 
    this fails in the 
    $\mathrm{Sym}-\mathrm{Skew}$ case; as such we expected 
    the problem to prove harder to solve: surprisingly, this turns out not to be the case. In other words, the equality between the quasiconvex and
rank-one convex bounds appears, a priori, more delicate in the
$\mathrm{Sym}$--$\mathrm{Skew}$ setting. The
martingale--laminate correspondence developed below shows,
however, that once the sharp rank-one convex bound is identified,
this equality would follow from the Iwaniec conjecture (see Remark \ref{rmk:iwaniecreal}
for more details).
\end{remark}
In higher dimensions, we do not have a similarly
strong theorem, but we can prove a dimension-free
upper bound:
\begin{theorem}\label{thm:Korn}
Given $u\in \dot{W}^{1,p}(\R^d, \R^d)$ we have
\begin{equation*}
    \|\mathcal A (u)\|_{L^p}\le 
    \sqrt 3(p^*-1)\|\mathcal E(u)\|_{L^p}.
\end{equation*}
In other words
\begin{equation}
C(p,d)\le \sqrt 3(p^*-1)
\end{equation}
Moreover, 
\begin{equation}
C(p,d)\ge p^*-1.
\end{equation}
The same result holds true if we consider the
trace-free version of Korn's inequality for $d\ge 3$:
\begin{equation}
    p^*-1\le C_0(p,d)\le \sqrt{3}(p^*-1).
\end{equation}
Finally, if $u$ is generalised-radial,
meaning that $u(x)=r(|x|)x$ for some
$r\colon\mathbb{R}^+\to M_d(\mathbb{R})$, then
we have the following stronger result:
\begin{equation}
\|\mathcal A(u)\|_{L^p}
\le (p^*-1)\|\mathcal E(u)\|_{L^p}
\end{equation}
and, if $p\ge 2, d=2$, the constant cannot be improved.
\end{theorem}

These results naturally suggest the conjecture
that \emph{$p^*-1$ is in fact the sharp constant
in all the above inequalities}. This
is in line with the Iwaniec conjecture
about the norm of 
the Beurling-Ahlfors operator
and its generalisation to higher dimensions.
As we will show in section~\ref{sec:calcvar},
the above conjecture would follow if one could show
that a certain rank-one convex integrand
is quasiconvex at $0$, just as in the
case of Beurling-Ahlfors.
These $L^p$ estimates combine naturally
with a version of the Riesz-Thorin
theorem for Orlicz spaces obtained
in~\cite{RieszThorinOrlicz}.
As a consequence, we obtain a dimension-free
bound for the Korn
inequality in all Orlicz spaces
in which the inequality holds:
\begin{theorem}\label{thm:KornOrlicz}
Given an $N$-function $\Phi$ satisfying
both the $\Delta_2$
and $\nabla_2$ conditions
and $u\in \dot{W}^1L^\Phi(\R^d, \R^d)$ we have
\begin{equation*}
\|\mathcal A(u)\|_{L^\Phi}\le
K \max(p^*-1,q^*-1)\|\mathcal E(u)\|_{L^\Phi},
\end{equation*}
where $p,q$ are $\Phi$'s indices
(see section~\ref{sec:orlicz})
and $K$ is a constant that can be explicitly estimated
and in particular is independent of $d$.
A modular version
of the inequality also holds.
\end{theorem}
Finally, combining our bounds with
recent results from~\cite{Petermichl}
on weighted estimates for
differentially subordinate martingales,
we obtain an $A_p$-weighted version of our estimates
(see~\cite{Kilpelainen1994} for more
information about weighted Sobolev spaces):
\begin{theorem}\label{thm:Kornweighted}
Let $p\in (1,\infty)$, $w\in A_p(\R^d)$ and 
$u\in \dot{W}^{1,p}(w\diff x, \R^d,\R^d)$.
Then
\begin{equation*}
    \|\mathcal A(u)\|_{L^p(w\diff x)}
    \le K (p^*-1)
    [w]_{A_p^\text{\upshape Heat}}^
    {\max\left(1,\frac1{p-1}\right)}
    \|\mathcal E(u)\|_{L^p(w\diff x)},
\end{equation*}
where $K$ is a constant that can be explicitly estimated
and in particular is independent of $d$.
\end{theorem}
Here $[\cdot]_{A_p^{\text{\upshape Heat}}}$ denotes
a quantity equivalent to the usual
Muckenhoupt characteristic
that will be defined below.\par
These upper bounds rest on new estimates for $R\otimes 
R$, a matrix-valued analogue of the second-order 
Riesz transform.
Although the improved Korn inequality
is of independent interest,
it also yields stronger
coercivity estimates when $p^*-1$ is sufficiently large,
and hence new existence results
for related variational problems.\par
The inequality admits several other forms
on bounded domains with
boundary conditions; see, for instance~\cite{payne1961korn}.
We do not discuss those variants here,
although our methods should also lead to improved
estimates in those settings.\par
\subsection{Preliminaries}
We begin by summarising the notation
used throughout the paper.
By $\widehat f$ and $\mathcal F f$, we
denote the Fourier transform of
the tempered distribution $f$, defined for $f\in L^1(\R^d)$
as
\begin{equation*}
\widehat f(\xi)
\mathrel{:}=
\int_{\R^d} f(x)\exp(-2\pi i \langle x,\xi\rangle)
\diff x.
\end{equation*}
We equip the space of $d\times d$ matrices over
$\mathbb{R}$,
$M_d(\mathbb{R})$, with the Frobenius inner product, 
defined as 
$\langle A,B\rangle=\mathrm{Tr}(AB^t)$.
The associated norm, the Frobenius norm,
will be denoted by $|A|$.
More generally, tensor products of 
Euclidean spaces
$\mathbb{R}^{d_1}\otimes \mathbb{R}^{d_2}
\otimes \dots \otimes\mathbb{R}^{d_n}$
will always be endowed with the induced Frobenius norm 
\begin{equation*}
|a|^2=\sum_{1\le i_1\le d_1}\dots
\sum_{1\le i_n\le d_n}|a_{i_1,\dots, i_n}|^2.
\end{equation*}
Given a subspace $X$ of $M_d(\R)$,
we write $P_X$ for the orthogonal
projection onto $X$. $\mathrm{Sym}(d)$
will denote the subspace of symmetric 
matrices in $M_d(\R)$, $\mathrm{Sym}_0(d)
$ that of trace-free symmetric matrices
and $\mathrm{Skew}(d)$ that of skew 
matrices. We will also need the following family of 
angle-preserving matrices:
\begin{equation*}
\mathrm{CO}(d)=\mathbb R\mathrm{O}(d)
=\{\lambda O:\lambda\in \R,\ O\in \mathrm{O}(d)\}
\end{equation*}
and its subsets $\mathrm{CO}^+(d)\mathrel{:}=
\mathrm{CO}(d)\cap \{A:\det(A)\ge 0\}$
and $\mathrm{CO}^-(d)\mathrel{:}=
\mathrm{CO}(d)\cap \{A:\det(A)\le 0\}$,
respectively the sets
of conformal and anti-conformal matrices.\par
We shall use several function spaces which we now define;
all of them will be understood to be defined
up to null sets.
$W^{1,p}(\Omega, \mathbb{R}^d)$ is 
the space of vector-valued Sobolev functions,
which we equip with the norm
\begin{equation*}
\|f\|_{W^{1,p}}^p=\int_{\Omega}|f(x)|^p\diff x
+\int_{\Omega}|\nabla f(x)|^p\diff x
\end{equation*}
with the usual modification for $p=\infty$.\par
The homogeneous Sobolev space
$\dot{W}^{1,p}(\Omega, \mathbb{R}^d)$ is 
defined as the space of $\R^d$-valued locally integrable
functions $f$
whose (weak) derivatives $\nabla f$ belong to
$L^p(\Omega, M_d(\mathbb{R}))$ and it is equipped with its (semi-)norm
\begin{equation*}
\|f\|_{\dot{W}^{1,p}}=
\left(\int_{\Omega}|\nabla f(x)|^p\diff x\right)
^{\frac 1p}.
\end{equation*}
Orlicz spaces are defined as follows:
given a Young function
$\Phi$ and a separable Hilbert space $X$, 
\begin{equation*}
L^\Phi(\Omega, X)
\mathrel{:}=
\left\{f\colon\Omega\to X\text{ measurable such that }
\exists K>0
\int_\Omega \Phi(K|f(x)|)\mathrm{d}x<\infty\right\},
\end{equation*}
which we consider equipped with the Luxemburg norm,
\begin{equation*}
\|f\|_{L^\Phi}\mathrel{:}=
\inf\left\{k \in (0,+\infty)\colon
\int_\Omega \Phi(|f(x)|/k)\mathrm{d}x<1 \right\}.
\end{equation*}
We refer the reader to \cite{analysisBS}
for more information on vector-valued
integration.
Orlicz-Sobolev spaces
are then defined in the natural way:
\begin{equation*}
W^1L^\Phi(\Omega, \R^d)
\mathrel{:}=
\left\{f\colon f\in L^\Phi(\Omega, \R^d),
\nabla f\in L^\Phi(\Omega, M_d(\R))\right\}
\end{equation*}
with the natural norm
\begin{equation*}
\|f\|_{W^1L^\Phi}\mathrel{:}=\|f\|_{L^\Phi}+
\| |\nabla f|\|_{L^\Phi},
\end{equation*} 
and similarly for homogeneous Orlicz-Sobolev spaces.
The weighted variants of the above Sobolev spaces are
obtained by simply substituting
$w(x)\diff x$ for $\diff x$.\par
Finally, $B_V(x_0,r)$ denotes the open ball 
centred at $x_0$ and having radius $r$ in a normed 
space $V$, $\mathbb S^{d-1}$
the unit sphere in $\R^d$ and $\omega_{d-1}$
its surface area.
\section{The setup and a first estimate}
The first step in our proof
is to relate Korn's inequality to 
a singular integral operator, as is often
done in the literature. The connection
we are interested in is given
by the following proposition.
\begin{proposition}
Let $R_j$ denote the Riesz transform in
the $j$-th direction and $R$ the Riesz transform vector.
Then
\begin{equation}\label{eq:RieszKorn1}
\mathcal A(u)=(R\otimes R)\mathcal E(u)
-((R\otimes R)(\mathcal E(u)))^t, 
\end{equation}
where the matrix operator $R\otimes R$ acts
on a matrix-valued function
$F\colon\mathbb{R}^d\to M_d(\mathbb{R})$ by the rule
\begin{equation}\label{eq:RieszKorndef}
((R\otimes R)(F))_{i,l}=\sum_j R_iR_jF_{j,l}.
\end{equation}
Similarly, if we consider the trace-free part of 
$\mathcal E(u)$, which we denote by $\mathcal E_0(u)$,
and $\mathcal A_0(u)$ the remaining part
of the gradient, i.e.
$\mathcal A_0(u)=\mathcal A(u)
+P_{\mathrm{span}(\mathrm{Id})}(\nabla u)$,
we have
\begin{equation}\label{eq:RieszKornTrace}
    2P_{\mathrm{Skew}} (R\otimes R (\mathcal E_0(u)))-
    \frac{d}{d-1}P_{\mathrm{span}(\mathrm{Id})}
    (R\otimes R(\mathcal E_0(u)))=\mathcal A_0(u)
\end{equation}
where $P_X$ denotes the orthogonal projection
on the subspace $X$ of $M_d(\R)$.
\end{proposition}
Variants of this result have of course been known
for a long time.
We include the calculation here for completeness.
A slightly different approach can be found
in~\cite{spector}.
\begin{proof}
Taking the Fourier transform of the left-hand side
of~\eqref{eq:RieszKorn1} we get
\begin{equation*}
    2\pi\frac i2
    \left(
        \widehat u\otimes \xi^t-\xi\otimes \widehat u^t
    \right)
\end{equation*}
since
\begin{equation*}
    \mathcal F(\nabla u)=2\pi \widehat u\otimes i\xi^t.
\end{equation*}
To calculate the Fourier transform
of the right-hand side
of~\eqref{eq:RieszKorn1}, first notice that
\begin{align*}
\mathcal F(R\otimes R (\mathcal E(u)))
&=
-\frac1{|\xi|^2}\xi\otimes \xi^t
\widehat{\mathcal E(u)}
=
-\frac1{|\xi|^2}\xi\otimes \xi^t
\left(
2\pi \frac i2
\left(
\xi\otimes \widehat u^t 
+\widehat u\otimes \xi^t
\right)
\right)
\\&=
-2\pi \frac i2\xi\otimes \widehat u^t-
2\pi \frac{i}{2{|\xi|^2}}
\xi\otimes \xi^t\langle \xi, \widehat u\rangle
\end{align*}
Inserting this into the right-hand side
of~\eqref{eq:RieszKorn1} proves
it. The identity~\eqref{eq:RieszKornTrace} is proved similarly.
\end{proof}
It follows from the previous 
proposition that to 
obtain our upper bound a viable strategy is to try
and bound from above $\|P_{\mathrm{Skew}(d)}R\otimes R\|_{
L^p(\mathbb R^d, \mathrm{Sym}(d))
\to L^p(\R^d, \mathrm{Skew}(d))}.$
This problem resembles closely the study
of norm bounds for the Beurling-Ahlfors operator
that for $d=2$ is defined as a
principal value integral by
\begin{equation}
    \label{eq:BA}
T_{BA}f(z)\mathrel{:}=\frac1{2\pi i}\int_{\C}
\frac{f(\zeta)}{(z-\zeta)^2}\diff \zeta
\wedge \diff \bar \zeta.
\end{equation}
For the $d$-dimensional definition,
we refer the reader to~\cite{IM}. 
In $2$ dimensions, the problem of calculating
the norm of $T_{BA}$ has been studied since the 1990s.
The problem was first introduced by Iwaniec and 
Martin, who proved in~\cite{IM}
\footnote{The result $\|T_{BA}\|_{L^p}\ge p^*-1$
was proved by Lehto in~\cite{Lehto_1966}.} that
$T_{BA}$, considered as an operator on $L^p$, 
satisfies
\begin{equation*}
p^*-1\le \|T_{BA}\|_{L^p\to L^p}\le 4 (p^*-1)
\end{equation*}
and conjectured that $\|T_{BA}\|=p^*-1$.
The result was subsequently improved by many authors
and at the time of writing
the best upper bound
currently known,
due to~\cite{Banuelos2008}, is 
$\|T_{BA}\|_{L^p}\le C (p^*-1)$
with $C\approx 1.575$
(see also~\cite{Volberg} 
for improved asymptotic bounds).
In higher dimensions, the problem has proved
more complicated and in that case the best bound,
obtained by Hytönen in~\cite{Hytonen},
is $(d+1)(p^*-1)/2$ and
in particular it is not dimension-free.
It might then come as a surprise that
a dimension-free bound for $R\otimes R$
can be achieved. We first
prove a slightly weaker result utilising
a method of rotations introduced by 
Iwaniec and Martin in \cite{IM}.
\begin{proposition}
\label{prop:rotationbound}
For $u\in \dot W^{1,p}(\R^d,\R^d)$
we have
\begin{equation}
\|\mathcal A(u)\|_{L^p}\le 
\sqrt d H_p(2)\|\mathcal E(u)\|_{L^p},
\end{equation}
where $H_p(2)=\|T_{BA}\|_{L^p(\C)\to L^p(\C)}$
and where $T_{BA}$ is defined in \eqref{eq:BA}.
\end{proposition}
\begin{proof}
We only consider the case $p\in [2,\infty)$.
The proof works very similarly in the dual case,
the only difference being the change in $N$,
defined below at~\eqref{N}.
Clearly, it suffices to prove that $R\otimes R $, 
when considered as an operator from $L^p(\C^d, M_d(\C))$
into itself,
is a bounded operator with norm bounded
by $\sqrt d H_p(2)$.
To see this, we first observe that
the operator $T\mathrel{:}=R\otimes R$ can
be written as
\begin{equation*}
T A(z)=
\frac{\Gamma(d+1)}{\pi^d}\int_{\C^d}
\frac{\zeta\otimes \zeta^t A(z-\zeta)}{|\zeta|^{2d+2}}
\text{d} \zeta
\end{equation*}
(see~\cite{IM}, section $2$ for a proof).
We can thus apply a suitable version of the
polarized method of complex rotations
such as the one in~\cite{IM}, proposition 7.1 to obtain
\begin{equation*}
\|T\|_{L^p(\C^d, M_d(\C))\to L^p(\C^d, M_d(\C))}
\le \frac{H_p(2)}{2}
\sqrt{\omega_{2d-1}}\frac{\Gamma(d+1)}{\pi^d}\sqrt N,
\end{equation*}
where\footnote{For $p\in (1,2)$,
the expression for $N$ remains
the same except for the order of
$M$ and $M^\dagger$.}
\begin{equation}\label{N}
N\mathrel{:}=
\left\|
\int_{\mathbb S^{2d-1}} M_{\zeta\otimes \zeta^t}
M_{\zeta\otimes \zeta^t}^\dagger\text{d}\zeta
\right\|_{\mathcal L(M_d(\C), M_d(\C))}
\end{equation}
and $M_A$ denotes the operator on $M_d(\C)$
defined as $M_A(B)\mathrel{:}=AB$.
To calculate $N$, we first notice
that $M_A^\dagger=M_{A^\dagger}$
and that $M_AM_B=M_{AB}$,
thus the result of the integration
will be the operator $M_C$ with 
\begin{equation*}
C\mathrel{:}=
\int_{\mathbb{S}^{2d-1}}\zeta\otimes
\zeta^t(\zeta\otimes \zeta^t)^\dagger
\text{d}\zeta=
\int_{\mathbb{S}^{2d-1}}
\zeta\otimes\zeta^t(\zeta^t)^\dagger\otimes \zeta^\dagger
\text{d}\zeta=\int_{\mathbb{S}^{2d-1}}
\zeta\otimes \zeta^\dagger\text{d}\zeta.
\end{equation*}
The right-hand side can be written component-wise as
\begin{equation*}
C_i^j=\int_{\mathbb{S}^{2d-1}}
\zeta_i\bar{\zeta_j}\text{d}\zeta,
\end{equation*}
which is $0$ unless $i=j$, in which case we get
\begin{equation*}
C_i^i=\int_{\mathbb{S}^{2d-1}}
|\zeta_i|^2\text{d}\zeta
=\omega_{2d-1}\frac{1}{d}.
\end{equation*}
Thus $C=1/d\omega_{2d-1}\mathrm{Id}$, so 
\begin{equation*}
N=\|M_C\|=\frac{1}{d}\omega_{2d-1}\|M_{\mathrm{Id}}\|=
\frac{1}{d}\omega_{2d-1}.
\end{equation*}
Thus
\begin{equation*}
\|T\|\le \frac{H_p(2)}2\frac{d!}{\pi ^d}\omega_{2d-1}
\frac{1}{\sqrt{d}}=\frac{H_p(2)}2
\frac{d!}{\pi^d}\frac{2\pi^d}{(d-1)!}\frac{1}{\sqrt{d}}
\le \sqrt{d}H_p(2)
\end{equation*}
\end{proof}
\begin{remark}
The dimension-dependence in this estimate is
a consequence of the matricial nature of our operator $T$.
Indeed, if instead one considers the operator
$R\otimes R$ mapping $L^p(\C^d,\C)$
into $L^p(\C^d, M_d(\C))$,
using the method described above one gets
a bound of the form $\lesssim H_p(2)$ (see~\cite{IM}).
\end{remark}
\section{The toolkit: Burkholder's method and calculus of 
variations}\label{sec:prereq}
While the method of complex rotations
is usually very efficient, it often does not
lead to the best constants. Again taking inspiration
from the Beurling-Ahlfors case,
in order to improve the estimates one usually turns
to Bellman functions
(see~\cite{Vasyunin_Volberg_2020}) or 
Burkholder's estimates (see~\cite{Banuelos}).
We will follow
the second approach; we refer the reader to Remark
\ref{rmk:bellmanisotropic} for more on a potential 
Bellman function-based approach.
For more on the stochastic background,
see~\cite{EvansSDE} and references therein.
We will in the remainder of this work make extensive use of 
several results concerning Burkholder's theory of differential 
subordination. Since readers coming from the calculus of 
variations may be less familiar with this theory, we briefly
review the necessary background. Excellent overviews of the
theory, which do it much more justice than we can here, are also 
given in the introductory Chapter
of~\cite{osekowski} and in~\cite{Banuelos}.
\subsection{A brief overview of Burkholder's theory}    
\begin{definition}[Differentially subordinate martingales]
Let $(X,\Sigma, \mu)$ be a $\sigma$-finite measure space
and let $(\mathcal F_n)_{n\in \N}$ be a filtration such that
$\mu$ is $\sigma$-finite on $\mathcal F_0$.
Let $H$ be a separable Hilbert space and let
$(X_n)_{n\in \N},(Y_n)_{n\in \N}$ be two 
$H$-valued martingales.
We say that $X$ is $p$-differentially subordinate
to $Y$ if $|X_0|\le (p^*-1)|Y_0|$ and
$|dX_n|^2\le |dY_n|^2$ almost everywhere
 for all $n\ge 1$, where 
$dX_n=X_{n}-X_{n-1}$.
\end{definition}
\begin{theorem}[{\cite[Theorem 3.5]{osekowski}},
\cite{Burkholder1988Sharp}]
\label{thm:burk}
Consider $(X_n)_{n\in \N}$ and $(Y_n)_{n\in \N}$
two martingales such that
$X$ is $p$-differentially subordinate to $Y$. Then
\begin{equation}
\|X_n\|_{L^p}\le (p^*-1)\|Y_n\|_{L^p}.\label{eq:burk}
\end{equation}
Moreover, the inequality is sharp
even among dyadic pairs $((X_n),(Y_n))$
where $Y_n$ is real-valued with $Y_0=0$ and
$X_n$ is a $\pm1$ transform of $Y_n$,
meaning that
$\mathrm{d}X_n=\varepsilon_n \mathrm{d}Y_n$
with $\varepsilon_n$
a predictable sign sequence.
\end{theorem}
The main ingredient necessary to prove
this theorem is the construction
of an auxiliary function $U_p(x,y)\colon H^2\to \mathbb R$ 
satisfying
\begin{enumerate}
\item $U_p(x,y)\le (p^*-1)^p|y|^p-|x|^p=:V_p(x,y)$
\item For any $x,y,h,k\in H$ with 
$|h|\le |k|$
\footnote{Perhaps surprisingly, it turns out that one
obtains the same function when studying
the problem restricted
to martingale transforms, 
i.e. $|h|=|k|$},
the function
\begin{equation*}
G(t)=U_p(x+th,y+tk)
\end{equation*}
is convex.
\end{enumerate}
Indeed, a direct computation shows that if $U_p$ satisfies
these properties, then for any $p$-differentially
subordinate pair $U_p(X_n,Y_n)$ is a submartingale.
Note that by convexity we have
\begin{equation*}
    U_p(X_n,Y_n)\ge U_p(X_{n-1},Y_{n-1})+
    \langle \partial_xU_p(X_{n-1},Y_{n-1}), dX_{n}\rangle+
    \langle \partial_y U_p(X_{n-1},Y_{n-1}), dY_{n}\rangle
\end{equation*}
and taking conditional expectation with respect to
$\mathcal F_{n-1}$ gives the claim.
Iterating we obtain
\begin{equation*}
    \mathbb E(V_p(X_n,Y_n))
    \ge \mathbb E(U_p(X_n,Y_n))
    \ge U_p(X_0,Y_0)\ge 0.
\end{equation*}
Indeed, even if we only have $|dX_n|\le |dY_n|$, we 
still have
\begin{equation}
    \label{eq:genburk}
    (p^*-1)^p\|Y_n\|_{L^p}^p-\|X_n\|_{L^p}^p\ge U_p(|X_0|,|Y_0|)
\end{equation}
Constructing such a function is far from straightforward 
and constitutes a great accomplishment of Burkholder's.
In particular, Burkholder obtained the best
such function in~\cite{Burkholder1988Sharp}
(see~\cite{VasVol} for an alternative derivation):
\begin{equation}
    \label{eq:burkfunct}
    \mathcal G_p(x,y):=\begin{cases}p\left(1-\frac1{p^*}\right)^{p-1}
    ((p^*-1)|y|-|x|)(|x|+|y|)^{p-1}&\mathrm{if\ }V_p(x,y)\le 0\\
    V_p(x,y)&\mathrm{otherwise.}
    \end{cases}
\end{equation}
With this function in hand, the verification
of the required properties is routine
and the theorem is proved.
\begin{remark}
$\mathcal G_p$ remains the best possible such 
function even when we only consider dyadic $\pm1$
martingale transforms, see \cite{VasVol}. In
other words, if we denote by $\Sigma(x_0,y_0)$
the class of pairs of scalar-valued
dyadic martingales $(f_n,g_n)$
with $g_0=x_0,f_0=y_0$ with $|dg_n|=|df_n|$, then
\begin{equation}
    \inf_{\Sigma(x_0,y_0)}(p^*-1)^p\|f_n\|_{L^p}^p
    -\|g_n\|_{L^p}^p=\mathcal G_p(x_0,y_0).
\label{eq:maxdyad}
\end{equation}
\end{remark}
This result can be easily extended to
continuous time martingales,
under some assumption of regularity.
For details on these assumptions, and 
a definition of the quadratic variation
$\langle X\rangle_t$ of a continuous 
martingale
we refer the reader to~\cite[Chapter 1]{Shreve}
and~\cite[Chapter IV]{revuzyor1999}.
\begin{definition}
Let $(X,\Sigma, \mu)$ be a $\sigma$-finite measure space
and let $(\mathcal F_t)$ be a filtration
satisfying the usual assumptions such that $\mu|_{\mathcal F_0}$
is $\sigma$-finite. Then given a separable Hilbert space $H$
and two continuous martingales
$X,Y$ taking values in $H$,
we say that $X$ is differentially subordinate
to $Y$ if
$\mathrm d\langle X\rangle_t\le
\mathrm d\langle Y\rangle_t$ almost everywhere
and $X_0=Y_0=0$.
\end{definition}
Burkholder's theorem can then be extended
(by applying Itô's formula)
to such pairs of differentially
subordinate continuous martingales,
whereby we get
\begin{theorem}[{\cite[Theorem 5.7]{osekowski}}]
\label{thm:burkholder}
Let $(X_t), (Y_t)$ be two continuous $H$-valued
martingales with $H$ a separable Hilbert space. If $X$
is differentially subordinate to $Y$, then
\begin{equation}
\vvvert{X}_{L^p}\le (p^*-1)\vvvert{Y}_{L^p},
\label{eq:Burkholder}
\end{equation}
where $\vvvert{X}_{L^p}\mathrel{:}=
\sup_t \|X_t\|_{L^p}$.
Moreover, the following stronger claim is true:
for each time $t$ we have
\begin{equation}
\|X_t\|_{L^p}\le (p^*-1)\|Y_t\|_{L^p}.
\label{eq:Burkholderpointwise}
\end{equation}
\end{theorem}
Similarly, though with different functions,
one can prove a plethora of inequalities.
For example (\cite[Theorem 3.2, Theorem 5.7]{osekowski})
we can bound the weak $L^p$ norm by
\begin{align*}
\|X_t\|_{L^{p,\infty}}\le
    \begin{cases}
    \left(\frac{2}{\Gamma(p+1)}\right)^{\frac1p}
    \|Y_t\|_{L^p}
    &\text{if $p\in(1,2]$}\\ 
    \left(\frac{p^{p-1}}2\right)^{\frac1p}
    \|Y_t\|_{L^p}&\text{if $p>2$}
    \end{cases}
\end{align*}
These results are very similar to results concerning
boundedness for singular integral operators.
Inspired by this connection and by
Hytönen's solution to the $A_p$
conjecture, one may wonder if a "weighted"
Burkholder inequality holds.
Such a result was recently obtained by Domelevo,
Petermichl and Skreb
\cite{Petermichl} (using a different technique
from the one outlined above).
Before stating it, let us recall the definition
of a martingale weight:
\begin{definition}\label{def:weightmart}
Let $(w_n)$ be a martingale associated to
the filtration $\mathcal F_n$.
We say $(w_n)$ belongs to $A_p$ for $p\in (1,\infty)$
if $\sup_n \|w_n\|_{L^p}<\infty$, $w_n$ is positive
almost everywhere and
\begin{equation*}
[w]_{A_p^{\text{mart}}}
\mathrel{:}=
\sup_n \left\|w_n \mathbb E(w^{-\frac1{p-1}}
|\mathcal F_n)^{p-1}
\right\|_\infty<\infty,
\end{equation*}
where $w=\lim_{n\to \infty} w_n$, which exists by standard martingale
convergence theorems.
If $(w_n)$ belongs to some $A_p$, we say it is a weight.
\end{definition}
With this definition and its natural extension
to the continuous case, the weighted
Burkholder theorem can be stated:
\begin{theorem}\label{thm:burkweight}
Let $(X_t), (Y_t)$ be two continuous martingales
on the same filtered space. 
Assume $X_0=Y_0=0$, and that
$X_t$ is differentially subordinate to $Y_t$. Then
\begin{equation*}
\sup_t \|X_t\|_{L^p(w\mathrm{d}\mu)}\lesssim
(p^*-1)[w]^{\max(1,\frac1{p-1})}_{A_p^{\mathrm{mart}}}
\sup_t \|Y_t\|_{L^p(w\mathrm d \mu)},
\end{equation*}
where we define $L^p(w)$ on martingales as:
\begin{equation*}
\vvvert{X}_{L^p(w)}=
\sup_t \mathbb E(|X_t|^pw_t)^{\frac1p}.
\end{equation*}
\end{theorem}
The implied constant in the theorem can be
made explicit and is
in particular independent of dimension.

The technique we will use in 
the following to obtain upper bounds
for Korn's inequalities is 
similar to the technique pioneered by Bañuelos and Wang
in~\cite{Banuelos1995SHARP2po}
in the context of the Iwaniec conjecture (sometimes referred to as
the Iwaniec-Martin conjecture).
Namely, one defines via heat extension
for each $T>0$ a martingale 
$M_t=P_{T-t}\nabla u(B_t),0\le t<T$ (a slightly
peculiar Brownian motion is 
required, see the next section for details)
associated to the gradient $\nabla u$ and then considers
its projections: in the Iwaniec case, these are 
$X_t=P_{\mathrm{CO}^+(2)}(M_t)$,
$Y_t=P_{\mathrm{CO}^-(2)}(M_t)$.
If $Y_t$ is differentially subordinate to $X_t$
(up to a multiplicative constant), Burkholder's method
implies $L^p$ bounds for the martingales,
which then in the limit imply the $L^p$ bounds
for $\nabla u$ and its projections.
\subsection{The viewpoint of calculus of variations}
We briefly recall some concepts from
calculus of variations. This subsection is not
strictly necessary to 
understand the proof of Theorem~\ref{thm:Korn}.
The most important 
concepts we need are
two notions of generalised convexity
for functions $f\colon M_d(\R)\to \R$,
henceforth referred to as integrands.
\begin{definition}[Rank-one convexity]
An integrand $f\colon M_{d}(\mathbb{R})\to
\mathbb{R}$ is rank-one convex
if it is convex on each rank-one line, i.e. on each 
segment connecting two matrices $A,B$ such that
$\mathrm{rank}(A-B)=1$. $f$ is rank-one convex at $A_0$
if there exists a rank-one convex integrand $g$
such that $g\le f$ and $g(A_0)=f(A_0)$.
\end{definition}
\begin{remark}
A rank-one convex integrand is locally
Lipschitz (\cite[lemma 5.6]{Rindler}).
\end{remark}
\begin{definition}
A continuous integrand $f\colon M_d(\R)\to \R$ is quasiconvex
if for all vector fields 
$u$ that belong to
$W^{1,\infty}_0(B_{\mathbb{R}^d}(0,1),
\mathbb{R}^d)$ and $A\in M_d(\R)$
we have
\begin{equation*}
\fint_{B_{\mathbb{R}^d}(0,1)}
f(A+\nabla u(x))\mathrm{d}x\ge f(A).
\end{equation*}
$f$ is quasiconvex at $A_0$ if there exists
a quasiconvex integrand $g$ such that
$g\le f$ and $g(A_0)=f(A_0)$.
\end{definition}
It is well known (see~\cite[Proposition 5.3]{Rindler}
and~\cite[Lemma 4.3]{Muller1999VariationalTransitions})
that quasiconvex integrands are rank-one convex.
The opposite implication, known as Morrey's problem,
has attracted considerable attention as one of 
the main open questions in the field.
Thanks to a remarkable result of
Šverák in~\cite{Sverak1992Rank-oneQuasiconvexity}
(see also~\cite{Grabovsky2017}), 
we now know the two conditions are not equivalent
in dimensions higher than $2$. The difference
between the two conditions is, in fact, substantial: 
in dimensions higher than $2$
quasiconvexity is a non-local condition
(see~\cite{Kristensen1999-Nonlocality}).
In dimension $2$\footnote{The only remaining 
cases are for integrands defined
on $2\times n$ for $n\ge 2$:
if the test functions are scalar,
all these notions of generalised
convexity collapse into convexity.},
Morrey's problem remains unanswered, though several
positive results are known for 
particular functions and particular domains
(see~\cite{Harris2018},~\cite{Faraco2008},~\cite
{Muller1999Rank-oneMatrices}).
Note in particular that the non-locality of
quasiconvexity remains an open problem here
and indeed if it were a local condition, then
it would coincide with rank-one convexity,
see~\cite[Theorem 2]{Kristensen1999-Nonlocality}.\par
As with most notions of convexity, we also have 
an associated notion of envelope;
the rank-one convex and quasiconvex envelopes
of an integrand $f\colon M_{d}(\R)\to \mathbb{R}$ 
are defined respectively as
\begin{equation*}
    f^{\mathrm{rc}}(A)=\sup\{g(A)\colon 
    g\le f\ \mathrm{and}\ g\ \text{rank-one convex}\}
\end{equation*}
and
\begin{equation*}
f^{\mathrm{qc}}(A)\mathrel{:}=
\sup\{g(A)\colon  g\le f\ \mathrm{and}\ g\ 
\mathrm{ quasiconvex}\}.
\end{equation*}
While these definitions are not particularly useful
for concrete calculations, 
the envelopes admit an alternate description by duality,
which is the description we shall find most useful.
For $f^{\mathrm{rc}}$ we have (see 
\cite[Theorem 6.10]{Dacorogna2008DirectVariations}
for the formula for the envelope,
\cite{Pedregal1993LaminatesMicrostructure}
and~\cite[Section 9.1]{Rindler} for
definitions and generalities on laminates):
\begin{equation}\label{eq:rcenv}
f^{\mathrm{rc}}(A)
=\inf \left\{\langle f,\mu\rangle\colon
\mu\in\mathcal M_A\right\},
\end{equation}
where $\mathcal M_A$ represents
the set of laminate measures (of finite order) having
$A$ as their barycentre.
Likewise for the quasiconvex envelope
we have Dacorogna's formula
(\cite[Theorem 6.9]
{Dacorogna2008DirectVariations},~\cite[Appendix]
{Kinderlehrer1991CharacterizationsGradients})
\begin{align}
\label{eq:qcenv}
    f^{\mathrm{qc}}(A)
&=
    \underset{
    u\in W^{1,\infty}_0(B_{\mathbb{R}^d}(0,1),
    \mathbb{R}^d)
    }
\inf
    \frac{1}{|B_{\mathbb{R}^d}(0,1)|}
    \int_{B_{\mathbb{R}^d}(0,1)}f(A+\nabla u(x))
    \mathrm{d} x\\ 
    &={\inf}
    \{\langle f,\mu\rangle\colon \mu\in
    \mathbf{GYM}_A\},
\end{align}
where $\mathbf{GYM}_A$ denotes the set of
homogeneous gradient Young measures
having barycentre $A$ (see for instance \cite{Rindler}).
We shall need the following elementary lemma:
\begin{lemma}
    Given a continuous integrand
    $f\colon M_d(\R)\to \R$, either
$f^{\mathrm{qc}}\equiv-\infty$ or 
$f^{\mathrm{qc}}>-\infty$ and 
$f^{\mathrm{qc}}$ is quasiconvex, 
and similarly for rank-one convexity.
If $f$ is positively $p$-homogeneous for some $p$, meaning
$f(tx)=t^p f(x)$ for all $x\in M_d(\R)$ and all $t>0$, then either
$f^{\mathrm{qc}}(0)=0$ or $f^{\mathrm{qc}}(0)=-\infty$.
\end{lemma}
It follows that $C(p,d)$ is equal
to the minimal constant
$C^{\mathrm{qc}}(p,d)$ which makes the integrand
\[
f(A):=C^p|P_{\mathrm{Sym}(d)}(A)|^p-
|P_{\mathrm{Skew}(d)}(A)|^p\]
quasiconvex at $0$. In light of Morrey's problem,
studying the relation between $C^{\mathrm{qc}}$
and $C^{\mathrm{rc}}$\footnote{the minimal constant that
makes the function rank-one convex at $0$}
is particularly interesting: either we have 
$C^{\mathrm{qc}}=C^{\mathrm{rc}}$, in which case
the inequality holds for the best constant
one could hope for,
or $C^{\mathrm{qc}}>C^{\mathrm{rc}}$, and by taking
rank-one convex envelope
we end up with a rank-one convex integrand
which is not quasiconvex.
\begin{remark}
The Iwaniec conjecture can also be seen
through these lenses,
with $\mathrm{Sym}(d)$ and $\mathrm{Skew}(d)$
substituted by $\mathrm{CO}^+(2)$ and
$\mathrm{CO}^-(2)$.
The link between sharp constants
and the calculus of variations
emerged prominently in Iwaniec's work
on the Beurling-Ahlfors operator
and generalisations in quasiconformal map theory,
see~\cite{Iwaniec2002Nonlinear},~\cite{Iwaniec1982}
and~\cite{Iwaniec1982OnExponents}.\par Let us denote by 
$\mathfrak C^{\mathrm{qc}}(p)$
and $\mathfrak C^{\mathrm{rc}}(p)$
the constants associated
with the Iwaniec problem.
The result of Lehto (\cite{Lehto_1966})
we mentioned before can now be
equivalently stated as
$\mathfrak C^{\mathrm{rc}}(p)\ge p^*-1$,
and so the Iwaniec conjecture
can also be written as
\[\mathfrak{C}^{\mathrm{rc}}=
\mathfrak{C}^{\mathrm{qc}}.\]
\end{remark}
We have described in the previous subsection how 
to estimate $C^{\mathrm{qc}}$ (from above),
but how does one
estimate $C^{\mathrm{rc}}$? Lehto's proof 
takes advantage of the complex-analytical nature of the
Iwaniec question\footnote{A proof that avoids
complex analysis
was obtained in~\cite{boros2011}.},
which is not available in this case,
especially if we try and work in general dimension $d$.
This question, studied in slightly
more generality in~\cite{Cassese2025},
turns out to have a strikingly similar answer
to the one for
$C^{\mathrm{qc}}$ described above.
In the rank-one convex case,
we are dealing only with pre-laminate measures,
which turn out 
to be exactly the laws of
matrix-valued dyadic martingales 
with rank-one increments: hence, 
the problem of determining 
$C^{\mathrm{rc}}$ becomes a martingale problem,
and Burkholder's method provides the tool necessary to
obtain bounds on the constant
and even to calculate it exactly
\footnote{Though the situation is not, in general,
as simple as in Korn's case,
see~\cite{Cassese2025} for details}.
\par
The parallel between the
two instances is particularly surprising given that
quasiconvex and
rank-one convex integrands can have,
a priori, wildly different behaviours.
While we do not have a satisfying "explanation"
of this phenomenon,
we will try here to provide an attempt 
at a rough heuristic: 
to estimate $C^{\mathrm{qc}}$,
it is necessary to extend the correspondence 
\[\{\text{pre-laminates}\}\leftrightarrow
\{\text{Martingales with rank-one increments}\}\]
to a dense subclass of
homogeneous gradient Young measures.
The heat 
martingale representation described above provides 
such an extension,
since $\mathrm{Law}(M_t^T)$ coincides with the
pushforward through
$(P_{T-t}\nabla u)$ of the Lebesgue measure
$\mathrm{d}x$ and as $t\to T$ we 
recover the measure
$\nabla u_{\#}(\mathrm{d}x)$. However,
simply considering $M_d(\R)$-valued heat-based
martingales is too naïve a representation:
it covers more measures than required,
leading to bounds that are
conjectured not to be optimal. In other words,
in the rank-one convex case not only
can we represent pre-laminates as
$M_d(\R)$-valued dyadic martingales,
but we know exactly which of those are
induced by pre-laminates;
a similar characterisation is lacking for 
homogeneous gradient Young measures.
This is not particularly surprising, 
given that (as one can see by the envelope formulas
mentioned above) Morrey's problem 
is equivalent to whether all homogeneous gradient
Young measures are laminates.\par
The progress obtained in the study
of the Iwaniec conjecture 
illustrates this rather well:
once a suitable martingale extension 
process has been found
(such as the heat-based one),
the subsequent work improving those bounds
has been focused on leveraging some of the properties
that the heat-martingales induced by gradients had,
essentially restricting
the class of martingales considered
(see for example~\cite{Volberg},~\cite{Banuelos2008}).
The current best known bound of $1.575(p^*-1)$
takes advantage of the fact 
that the martingales induced by these gradients
(or to be more specific their
conformal and anti-conformal projections)
are conformal martingales. In a sense then, these
works on the Iwaniec conjecture
can be seen as successive attempts at slimming down
the family of measures that are considered,
resulting in improved bounds
\footnote{Of course, restricting the family
of measures is only
half the battle, as one then needs to prove that
in this new and smaller class techniques à la Burkholder
and/or Bellman function
lead to an improved bound.}.\par
In this context, Burkholder's function $\mathcal G_p$ is
particularly important, acting
as a bridge between the discrete
and continuous settings. It is 
then worthy of mention that
its quasiconvexity is still an
open problem 
(see~\cite
{Astala2023TheTheory},~\cite
{KariAstala2024},~\cite
{Iwaniec2002Nonlinear},~\cite
{AstalaIwaniecPrauseSaksman2015},~\cite
{Guerra2019} and~\cite{BaernsteinII1997}).
\section{Proof of Theorem~\ref{thm:Korn}}
\subsection{Proof of the upper bound
in Theorem~\ref{thm:Korn}}
We will prove the theorem
by using a suitably tensorialised
version of the martingale
projections defined by Bañuelos and Baudoin
in~\cite{BanuelosBaudoin2013}.
The proof can be reworked
to make no mention of martingales and
instead use the heat semigroup
directly (as
in~\cite{Petermichl2002HEATINGQUASIREGULAR}),
but this obscures the nature of the argument.
Before proving the result, let us recall
the setting introduced in~\cite{BanuelosBaudoin2013}.
Let $\mathcal B_t$ denote a Brownian motion in $\R^d$
with $\mathrm{d}x$ as its starting distribution,
and let for $0\le t\le T$
$Y^T_t\mathrel{:}=(\mathcal B_t, T-t)$
(henceforth referred to as
the space-time Brownian motion).
Denote the heat semigroup by $(P_t)_{t\ge 0}$.
Given $f\in C^\infty_c(\mathbb{R}^d)$,
we define the process
\begin{equation*}
  M_t^T \mathrel{:}=
 Uf(Y_t^T)=P_{T-t}f(\mathcal B_t),\qquad 0\le t\le T,
\end{equation*}
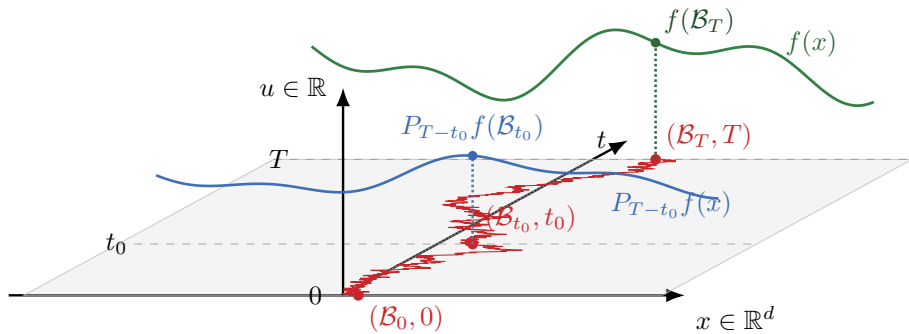
\begin{figure}
    \centering
\begin{tikzpicture}[
  x={(1.00cm,0cm)},
  y={(0.62cm,0.34cm)},
  z={(0cm,0.95cm)},
  >=Latex,
  font=\small
]
  \definecolor{timeblue}{RGB}{57,106,177}
  \definecolor{pathred}{RGB}{204,37,41}
  \definecolor{heatgreen}{RGB}{62,150,81}
  \definecolor{softgray}{RGB}{110,110,110}

  \def\Tmax{5.3}
  \def\Tplane{5.3}
  \def\tt{2}

  \draw[->, line width=0.9pt] (-4.4,0,0) -- (4.5,0,0) node[below right] {$x\in\mathbb{R}^d$};
  \draw[->, line width=0.9pt] (0,0,0) -- (0,\Tmax+0.75,0) node[left=4pt] {$t$};

   \fill[softgray!25, fill opacity=0.28, draw=softgray!45, draw opacity=0.8] (-4.2,0,0) -- (4.2,0,0) -- (4.2,\Tplane,0) -- (-4.2,\Tplane,0) -- cycle;
  \draw[softgray!55, dashed] (-4.2,\tt,0) -- (4.2,\tt,0);
  \draw[softgray!65, dashed] (-4.2,\Tmax,0) -- (4.2,\Tmax,0);
  \draw[softgray!55, densely dotted] (0,0,0) -- (0,\Tmax,0);
  \node[left=4pt] at (-3.8,\tt,0) {$t_0$};
  \node[left=2pt] at (-3.8,\Tmax,0) {$T$};
  \node[left=4pt] at (0,0,0) {$0$};
  \draw[->, line width=0.9pt] (0,0,0) -- (0,0,2.9) node[left=2pt] {$u\in \R$};

  \pgfmathsetseed{11112017}
  \pgfmathtruncatemacro{\nsteps}{1000}
  \pgfmathsetmacro{\xoffset}{0.20}
  \pgfmathsetmacro{\xpathscale}{0.80}
  \pgfmathsetmacro{\B}{0}
  \pgfmathsetmacro{\dt}{\Tmax/\nsteps}
  \pgfmathtruncatemacro{\kslice}{floor(\nsteps*\tt/\Tmax)}
  \xdef\BrownianPath{(\xoffset,0,0)}
  \xdef\BxSlice{\xoffset}
  \xdef\BxFinal{\xoffset}
  \foreach \k in {1,...,\nsteps}{
    \pgfmathsetmacro{\xi}{sqrt(12)*(rnd-0.5)}
    \pgfmathsetmacro{\Bnext}{\B + sqrt(\dt)*\xi}
    \xdef\B{\Bnext}
    \pgfmathsetmacro{\tk}{\k*\dt}
    \pgfmathsetmacro{\xk}{\xoffset + \xpathscale*\B}
    \xdef\BrownianPath{\BrownianPath (\xk,\tk,0)}
    \ifnum\k=\kslice
      \xdef\BxSlice{\xk}
    \fi
    \ifnum\k=\nsteps
      \xdef\BxFinal{\xk}
    \fi
  }

  \draw[pathred, line width=0.40pt] plot coordinates {\BrownianPath};
  \fill[pathred] (\xoffset,0,0) circle (2.0pt);
  \node[pathred, below right] at (\xoffset,0,0) {$(\mathcal B_0,0)$};
  \fill[pathred] (\BxFinal,\Tmax,0) circle (2.0pt);
  \node[pathred, above right] at (\BxFinal,\Tmax,0) {$(\mathcal B_T,T)$};

  \draw[timeblue, line width=1.0pt, opacity=0.96]
    plot[domain=-3.7:3.7, samples=150, smooth, variable=\x]
      (\x,\tt,{0.95 + 0.22*sin(65*(\x+0.6)) + 0.10*cos(150*(\x-0.2))});

  \draw[heatgreen!85!black, line width=1.05pt]
    plot[domain=-3.7:3.7, samples=150, smooth, variable=\x]
      (\x,\Tmax,{1.35 + 0.38*sin(70*(\x+0.35)) + 0.22*cos(160*(\x-0.1))});

  \node[timeblue, anchor=west] at (2.15,\tt,0.55) {$P_{T-t_0}f(x)$};
  \node[heatgreen!70!black, anchor=west] at (2.4,\Tmax,1.62) {$f(x)$};

  \fill[pathred] (\BxSlice,\tt,0) circle (2.0pt);
  \node[pathred, above right=1pt] at (\BxSlice,\tt,0) {$(\mathcal B_{t_0},t_0)$};

  \pgfmathsetmacro{\Mval}{0.95 + 0.22*sin(65*(\BxSlice+0.6)) + 0.10*cos(150*(\BxSlice-0.2))}
  \pgfmathsetmacro{\Fval}{1.35 + 0.38*sin(70*(\BxFinal+0.35)) + 0.22*cos(160*(\BxFinal-0.1))}

  \draw[timeblue!90, densely dotted, line width=0.95pt]
    (\BxSlice,\tt,0) -- (\BxSlice,\tt,\Mval);
  \fill[timeblue] (\BxSlice,\tt,\Mval) circle (1.8pt);
  \node[timeblue, above =2pt] at (\BxSlice,\tt,\Mval) {$P_{T-t_0}f(\mathcal B_{t_0})$};

  \draw[heatgreen!70!black, densely dotted, line width=0.95pt]
    (\BxFinal,\Tmax,0) -- (\BxFinal,\Tmax,\Fval);
  \fill[heatgreen!70!black] (\BxFinal,\Tmax,\Fval) circle (1.8pt);
  \node[heatgreen!50!black, above right=-1pt] at (\BxFinal,\Tmax,\Fval) {$f(\mathcal B_T)$};
\end{tikzpicture}
\caption{Schematic of the heat martingale $P_{T-t}f(\mathcal B_t)$}
\end{figure}
where $Uf$ is the
heat extension of $f$ to $\R^d\times \R_{+}$.
Since $Uf(x,T-t)$ solves the backward heat equation
$\partial_t Uf+\tfrac12\Delta Uf=0$,
It\^o's formula implies that $(M_t^T)_{0\le t\le T}$
is a martingale and
\begin{equation}\label{eq:dM}
\mathrm dM_t^T = \sum_{k=1}^d \partial_k Uf(Y_t^T)\,
\mathrm d\mathcal B_t^k,
\end{equation}
where $\partial_k Uf$ is the derivative along $x_k$.
Finally, we can for any $i,j\in \{1,\dots, d\}$
define the martingale
\begin{equation*}
Z^T_t=\int_0^t \partial_j Uf(Y^T_s)\diff \mathcal B^i_s.
\end{equation*}
It is easy to check that
\begin{equation*}
\langle Z^T_t, Z^T_t\rangle
=\int_0^t |\partial_j Uf(Y^T_s)|^2\mathrm{d}s
\end{equation*}
and thus $\langle Z^T_t,Z^T_t\rangle
\le \langle M^T_t,M^T_t\rangle$. Moreover, one can prove
(see~\cite[Theorem 3.9.1]{Banuelos}) that 
\begin{equation*}
R_i R_j f(x)=
\lim_{T\to \infty}\mathbb E\left(
Z^T_T|Y^T_T=(x,0)\right),
\end{equation*}
with the convergence taking place in $L^p$.
This argument, together with Burkholder's
differential subordination result,
Theorem~\ref{thm:burkholder},
implies
\begin{align*}
\|R_i R_j f\|_{L^p}&\le \lim_{T\to \infty}
\left\|\mathbb E\left(Z^T_T|Y^T_T=(x,0)\right)
\right\|_{L^p}
\\ &\le
\limsup_{T\to \infty}\left\|Z^T_T\right\|_{L^p}
\\ &\le
(p^*-1)\limsup_{T\to \infty}\left\|M_T^T\right\|_{L^p}
\\ &\le 
(p^*-1)\|f\|_{L^p},
\end{align*}
which is exactly the kind of bound we are interested in.
To prove the last inequality, notice that
for $\phi\in C_0(\R^d)$
\begin{equation*}
    \int_{\R^d}\mathbb E_x(\phi(\mathcal B_t))\diff x=
    \int_{\R^d}\int_{\R^d}\phi(y)h_t(x,y)\diff y\diff x
    =\int_{\R^d}\phi(y)\diff y,\end{equation*}
where $h_t(x,y)$ denotes the transition probability
from $(x,0)$ to $(y,t)$, we used Fubini
and that $\int_{\R^d}h_t(x,y)\mathrm{d}x=1$.
Thus in particular $\mathcal B_t$ is distributed like
the Lebesgue measure and $\mathbb 
E\left(|f(\mathcal B_t)|^p\right)=\|f\|_{L^p}^p$,
so working with the 
martingales gives us the same norms as working with 
functions; this in particular implies that
the martingales considered
will be uniformly integrable (for details on
the theory of martingales
on $\sigma$-finite measure spaces, 
see~\cite[Chapter 3]{analysisBS},
~\cite[sections 39-43]{dellacherie}
and~\cite[chapters 5-6]{Stroock}).\par
We want to apply this framework to estimate $\|R\otimes
R\|$. For this, it will suffice to provide a 
higher-dimensional analogue of the above
martingale $Z^T$ so that it
is differentially subordinate to
$Uf(Y_t^T)$ for $f\in C^\infty_c(\R^d,\R^d)$,
where the heat extension is defined component-wise.
This is not a priori trivial, 
as the transition from scalar-valued to
matrix-valued functions
risks introducing some dimension-dependent 
factors,
as seen in Proposition ~\ref{prop:rotationbound}.
However, we can leverage the orthogonality
of the coordinates of $\mathcal B$ to obtain
the result.
\begin{proposition}\label{prop:upperbound}
For any $f\in L^p(\R^d, M_d(\R))$ we have
\begin{equation*}
    \|(R\otimes R) f\|_{L^p}\le (p^*-1)\|f\|_{L^p}.
\end{equation*}
\end{proposition}
\begin{proof}
To represent $R\otimes R (f)_{i,j}$,
we define the process
\begin{equation*}
\left(\mathcal Z^T_t\right)_{i,j}\mathrel{:}=
\sum_{k=1}^d\int_0^t \partial_i  U(f_{k,j})(Y_s^T)
\diff \mathcal B^k_s.
\end{equation*}
The advantage of this construction
(summing over the index of the Brownian motion, instead
of summing over $\partial_i$) is that we can now
leverage the orthogonality of the coordinates
of $\mathcal B$,
which implies
\begin{equation*}
\langle \mathcal Z^T,\mathcal Z^T\rangle=\sum_{i,j}
\sum_{k=1}^d \int_0^t |\partial_i U(f_{k,j})(Y_s^T)|^2
\mathrm{d}s
=\int_0^t |\nabla U(f)(Y_s^T)|^2\mathrm{d}s.
\end{equation*}
It follows that $\langle \mathcal Z^T,
\mathcal Z^T\rangle=\langle M^T,M^T\rangle$,
so $\mathcal Z^T$ is differentially subordinate
to $M^T$.
As one easily checks,
even in the matricial case we have convergence
in $L^p$ as $T\to \infty$
of $(R\otimes R)_T(f)(x)\mathrel{:}=
\mathbb E\left(\mathcal Z^T_T|Y_T^T=(x,0)\right)
$ to $R\otimes R(f)$. This allows us to obtain
\begin{align*}
\|R\otimes R(f)\|_{L^p(M_d(\mathbb C))}\le 
\lim_{T\to \infty}\vvvert{\mathcal Z^T}_{L^p}\le
\lim_{T\to \infty}(p^*-1)\vvvert{M^T}_{L^p}\le
(p^*-1)\|f\|_{L^p}.
\end{align*}
\end{proof}
\begin{remark}
The martingale formula we used is not the best
possible when one is concerned with
the scalar second-order Riesz transform $R_iR_j$:
in particular, taking $A=-e_{i,j}-e_{j,i}$
(where $e_{i,j}$ is the matrix
which is 1 in the $(i,j)$ 
position and $0$ everywhere else)
allows one to see that
$\|R_{i}R_j\|_{L^p}\le 1/2 (p^*-1)$ (and the 
inequality is actually attained,
see~\cite{Banuelos_Osekowski_2012}). However,
using this choice of $A$ makes it harder to use the
orthogonality properties of $\mathcal B_t$ and results
in an overall worse bound.
As we shall see presently, we still have
a linear algebra trick up our sleeve 
to improve the bound.
\end{remark}
This, however, is not enough to imply the upper bound
in Theorem~\ref{thm:Korn}. We first need to take
advantage of the problem's structure via
the following lemma:
\begin{lemma}\label{lemma:linalg1}
Let $a\in \R^d\otimes \mathrm{Sym}(d)$ be a tensor.
Then
\begin{equation*}
\sum_{i,j,k}|a_{i,j,k}-a_{j,i,k}|^2\le 3 |a|^2
\end{equation*}
and the constant $3$ is sharp.
\end{lemma}
\begin{proof}
Let $V=(\R^d)^{\otimes 3}$ and let $W$ be the subspace 
$W=\R^d\otimes \mathrm{Sym}(d)$.
We define the operators $\sigma, \tau$
on $V$ as follows: 
$\sigma(a)_{i,j,k}=a_{j,i,k}$ and 
$\tau(a)_{i,j,k}=a_{i,k,j}$ or
in other words $\sigma(a)(x,y,z)=a(y,x,z)$ and
$\tau(a)(x,y,z)=a(x,z,y)$. Then proving the
inequality is equivalent to proving that,
on $W$, we have
\begin{equation*}
|(I-\sigma)a|^2\le 3 |a|^2.
\end{equation*}
Since $|(I-\sigma)a|^2=2|a|^2-2\langle a,\sigma a\rangle
=2|a|^2-2\langle a, \pi\circ \sigma a\rangle$,
where $\pi$ is the projection from $V$ onto $W$,
it is enough to prove that the minimal eigenvalue
$\lambda_{\mathrm{min}}$ of the operator
$T\mathrel{:}=\pi\circ \sigma\circ \pi$
acting on $W$ is $-1/2$.
As one can check via some arithmetic, $T$ satisfies
\begin{equation*}
    T^2=\frac14(I+\tau)+\frac 12 T.
\end{equation*} Restricting this to $W$ we get
\begin{equation*}
    T^2=\frac 12 I +\frac 12 T,
\end{equation*}
i.e. $(T-\mathrm{I})(2T+I)=0$. The result follows,
since $T$ is self-adjoint and hence diagonalizable.
\end{proof}
\begin{proof}[Proof of the upper bound
in Theorem~\ref{thm:Korn}]
Instead of dealing with the martingale
representation of $R\otimes Rf$
and then applying the projection, we directly deal with
the martingale representation of
$R\otimes Rf-R\otimes Rf^t$,
with the added assumption that
$f\in L^p(\R^d, \mathrm{Sym}(d))$.
Namely, we set
\begin{equation*}
(Z^T_t)_{i,j}\mathrel{:}=
\sum_{k=1}^d \int_0^t \partial_i Uf_{j,k}(Y_s^T)-
\partial_j U f_{i,k}(Y^T_s)\mathrm d\mathcal B_s^k.
\end{equation*}
The quadratic variation of this process is
\begin{equation*}
\langle Z^T_t\rangle=
\int_0^t \sum_{k,i,j} |\partial_i Uf_{j,k}(Y_s^T)-
\partial_j U f_{i,k}(Y_s^T)|^2\mathrm ds.
\end{equation*}
Notice that the tensor
$a_{i,j,k}\mathrel{:}=\partial_i Uf_{j,k}$
is symmetric in the last two variables.
We can thus apply Lemma~\ref{lemma:linalg1}
to obtain
\begin{equation*}
\mathrm{d}\langle Z^T_t\rangle
\le 3 \mathrm{d}\langle M^T_t\rangle,
\end{equation*}
where $M^T_t$ is defined as above.
The rest of the proof follows by the same argument as in
Proposition~\ref{prop:upperbound}.
\end{proof}
By using a little more linear algebra, we can prove that something slightly stronger holds:
\begin{proposition}
    There exists a function $\mathfrak K(d)$ which is monotone
    increasing and such that $\lim_{d\to \infty}
     \mathfrak K(d)=1$
    such that
    \begin{equation*}
        C(p,d)\le \sqrt{3}\mathfrak K(d)(p^*-1).
    \end{equation*}
    More precisely,
    \[
    \mathfrak K(d)\mathrel{:}=\frac{\sqrt{d+2}(d-1)}{\sqrt{3d}+(d-2)\sqrt{d+2}}.
    \]
\end{proposition}
\begin{proof}
The bound will be improved
        by taking advantage of the different
        representations of the operator $\mathcal E(u)\mapsto\mathcal A(u).$ More precisely, let $\mathcal L_{s,t}:\R^d\otimes \mathrm{Sym}(d)\to \R^d\otimes \R^d\otimes \R^d$
        be defined as
        \begin{equation*}
            \mathcal L_{\alpha,\beta}(a)\mathrel{:}=\alpha(\mathrm{Id}-\sigma)(a)+(1-\alpha)J\circ \mathrm{tr}_{13}(a)+\beta J\circ \mathrm{tr}_{23}(a),
        \end{equation*}
        where $J:\R^d\to \R^d\otimes \mathrm{Skew}(d)$ is the natural
        inclusion $J(v):=\sum_{k=1}^d e_k\otimes (e_k\wedge v)$
        and $\mathrm{tr}_{ab}$ is the contraction on the coordinates 
        $a,b$ and $\alpha,\beta\in \R$. It is then not hard to see that
        defining the martingale
        \begin{equation*}
            \left(\mathcal Z^{T,\alpha,\beta}_t\right)_{i,j}\mathrel{:}=\sum_k\int_0^t \mathcal L_{\alpha,\beta}(\nabla U\mathcal E(u))(Y^T_s)_{i,j,k}\mathrm{d}
            \mathcal B^k_s
        \end{equation*}
        this also represents $\mathcal A(u)$,
        as one can prove by using the symmetry of 
        $\mathcal E(u)$ and $R_iR_j$.
        It follows that
        \begin{equation*}
            C(p,d)\le (p^*-1)\|\mathcal L_{\alpha,\beta}\|_{\mathrm{Op}}.
        \end{equation*}
        Taking the $\inf$ over $\alpha,\beta$ we obtain
        \begin{equation*}
            C(p,d)\le \sqrt 3\frac{\sqrt{d+2}(d-1)}{\sqrt{3d}+(d-2)\sqrt{d+2}}(p^*-1),
        \end{equation*}
        and this proves the result. Note that the right hand side
        is monotone increasing in $d$ and tends to $\sqrt 3$
        as $d\to \infty$, so this is not an asymptotic improvement.
    \end{proof}
In a very similar manner,
one can prove that\footnote{The same result holds
for $d=2$ as 
well, but the proof there is significantly different,
requiring the use of the induced complex structure of
the problem. This additional structure, if properly
taken advantage of, leads to even stronger bounds,
see~\cite{Volberg}.},
for $d\ge 3$
\begin{equation*}
\|\mathcal A_0(u)\|_{L^p(\R^d)}\le
\sqrt{3}(p^*-1)\|\mathcal E_0(u)\|_{L^p(\R^d)}.
\end{equation*}
Indeed, the exact same proof works, provided one
uses the following lemma in place of lemma 
\ref{lemma:linalg1}.
\begin{lemma}\label{lemma:linalg2}
Let $a\in \R^d\otimes \mathrm{Sym}_0(d)$.
Let $b$ be the tensor defined via
$b_{i,j,k}\mathrel{:}=a_{i,j,k}-a_{j,i,k}$
and $c_k$ the vector defined
as $c_k=\sum_i a_{i,i,k}$. Then
\begin{equation*}
\|b\|^2+\frac{d}{(d-1)^2}\|c\|^2\le C_d \|a\|^2,
\end{equation*}
where $C_2=4$ and for any $d>2$, $C_d=3$.
Moreover, the constant is sharp.
\end{lemma}
As in the case of $C(p,d)$, we can work with a slightly more
general representation, this time with an operator
$\mathcal L^0_\alpha$ depending on one parameter only.
This time, however, the improvement is much smaller
(the best improvement one obtains is $C_0(p,6)\le 1.64662(p^*-1)$)
and the closed formula is much more complicated, so we leave the 
calculations to the interested reader.
\begin{remark}
More generally, the above proof can be applied
to any function space where conditional expectations
are well-defined contractions
and a Burkholder-type result holds
(see for example~\cite[chapter 5]{osekowski}),
provided one can ensure
strong convergence as $T\to \infty$.
\end{remark}
\begin{remark}\label{rmk:bellmanisotropic}
These lemmas allow us to take full advantage
of the pointwise restrictions
we have on the range of $\mathcal E(u)$\footnote{Indeed, 
one can prove that no linear operator 
$\mathcal L$ can lead via $\mathcal Z^{\mathcal L}$
to a better bound than the one we have now.}.
The remaining constraints are
only visible at the level of first derivatives of
$\mathcal A(u), \mathcal E(u)$,
see~\cite{Amrouche2006},~\cite{Eastwood2000}.
In principle, it should be possible to use them
(once suitably reframed
in terms of the associated martingales
and their quadratic covariations) to 
define an improved Bellman function for the associated 
Burkholder problem, in a similar fashion as was done
for the Beurling-Ahlfors transform
with conformal martingales in~\cite{Volberg}.
However, in dimension $d>2$, 
the $\mathcal E_0\to \mathcal A_0$
transform loses its isotropy, which 
makes the problem much more difficult. In particular, it
is likely going to be difficult
to fully distinguish the class of heat 
martingales induced by $\mathcal E(u)$ and
$\mathcal A(u)$ by using only the martingales and their
quadratic variations, precisely because the
constraint is not isotropic.
\end{remark}
\subsection{Lower bounds}\label{sec:lowerbound}
We already know that
$\|R\otimes R\|_{L^p(M_d(\C))\to L^p(M_d(\C))}
\le p^*-1$.
It is not hard to prove that
$\|R\otimes R\|_{L^p(M_d(\C))}
\ge c_p\ge \max\left\{1/2(p^*-1),1\right\}$,
proving that our estimate is sharp up to the factor $2$.
\begin{proposition}
\label{prop:lowerboundriesz}
For $p\in(1,\infty)$ and $d\in \N\setminus\{0,1\}$,
\begin{equation*}
\max \left\{\frac12(p^*-1),1\right\}
\le c_p
\le \|R\otimes R\|_{
L^p(\R^d,M_d(\mathbb{R}))\to L^p(\R^d, M_d(\mathbb{R}))}
\le p^*-1.
\end{equation*}
\end{proposition}
\begin{proof}
The upper bound has already been obtained.
To prove the lower bound, recall that
$\|R_i^2\|_{L^p(\R^d)\to L^p(\R^d)}=c_p$
(see~\cite{Banuelos_Osekowski_2012}).
Given $\varepsilon>0$, choose
$f\in L^p\colon \|f\|_{L^p}=1$
and $\|R_1R_1 f\|\ge c_p-\varepsilon$. Then choosing
$M(x)$ defined as
\begin{equation*}
M(x)_{a,b}=\delta_1^a\delta_1^bf(x)
\end{equation*}
a straightforward computation proves that
$|R\otimes R (M)(x)|\ge |R_{1}R_1f(x)|$
and the claim follows by taking the $p$-norm
and letting $\varepsilon\to 0$.
\end{proof}
\begin{figure}[htbp] 
\centering 
\includegraphics[width=0.7\textwidth]{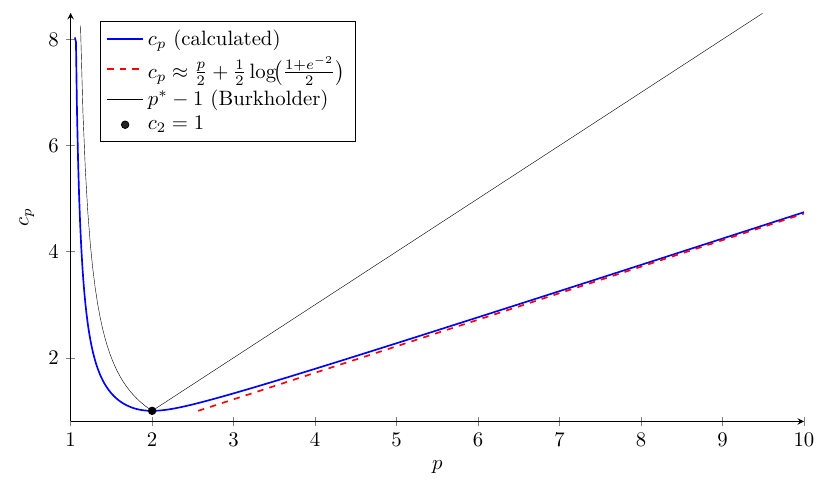}
\caption{Upper and lower bounds for
$\|R\otimes R\|_{L^p}$}
\label{fig:choi_plot} 
\end{figure}
In this section we will provide explicit
witnesses to the bound $C(p,d)\ge p^*-1$ for $p\ge 2$.
The same result will be proved for general $p$
via different methods,
in section~\ref{sec:calcvar}.\par 
Note that Proposition~\ref{prop:lowerboundriesz}
is not enough to obtain such a bound because
\begin{equation*}
\mathcal E(\dot W^{1,p}(\R^d,\R^d))
\subsetneq L^p(\R^d, \text{Sym}(d)).
\end{equation*} More precisely, 
one can show via Hodge decomposition that 
$\mathcal E(\dot{W}^{1,p}(\R^d,\R^d))$
is a nontrivial complemented
subspace of $L^p(\R^d,\mathrm{Sym}(d))$.
We thus need a different approach to obtain
a lower estimate. First, we reduce ourselves
to the $2$-dimensional case:
\begin{proposition}\label{prop:monotonedim}
For $d\in \N\setminus\{0,1\}$
and $p\in (1,\infty)$ the bound
\begin{equation*}
C(p,d+1)\ge C(p,d)
\end{equation*}
holds.
\end{proposition}
\begin{proof}
The result is clearly true for $p=2$
since $C(2,d)=1$ for all $d$,
so let us fix $p\in (1,\infty)\setminus \{2\}$.
We prove the result by induction on $d$.
Given $\varepsilon>0$,
let $f_{d}\colon\R^d\to \R^d$ be such that
\begin{equation*}
\frac{\|\mathcal A(f_d)\|_{L^p}}
{\|\mathcal E(f_d)\|_{L^p}}
\ge C(p,d)-\varepsilon.
\end{equation*}
Let $g\colon \R\to \R$ be a smooth function
such that $\text{supp}(g)\subset [-1,1]$
and $\|g\|_{L^p}=1$.
Then define 
\begin{equation*}
f_{d+1}(x_1,\dots, x_{d+1})=f_d(x_1,\dots, x_d)
g(x_{d+1}),
\end{equation*}
which is to be understood to be $\R^{d+1}$ valued via
the standard embedding $\R^d\hookrightarrow \R^{d+1}$.
If we write $C=2\|g'\|_{\infty}$, then we have
\begin{equation*}
\nabla f_{d+1}=
\begin{pmatrix}
g(x_{d+1})\nabla f_d(x_1,\dots, x_d)&
g'(x_{d+1})f_d(x_1,\dots,x_d)\\
0&0
\end{pmatrix},
\end{equation*}
hence 
\begin{equation*}
\|\mathcal A(f_{d+1})\|_{L^p}
\ge \|\mathcal A(f_d)\|_{L^p}
\end{equation*}
and
\begin{equation*}
\|\mathcal E(f_{d+1})\|_{L^p}
\le \|\mathcal E(f_d)\|_{L^p}+C\|f_d\|_{L^p},
\end{equation*}
hence
\begin{equation*}
C(p,d+1)\ge \frac{\|\mathcal A(f_d)\|_{L^p}}
{\|\mathcal E(f_d)\|_{L^p}+
C\|f_d\|_{L^p}}.
\end{equation*}
To get the result it is then just a matter of
taking $f_{d+1}^r\mathrel{:}=
f_d(rx_1,\dots, rx_d)g(x_{d+1})$
and letting $r\to \infty, \varepsilon\to 0$.
\end{proof}
In $2$ dimensions, the situation simplifies
considerably and we are
able to construct a sequence witnessing
the bound $C(p,2)\ge p-1$.
\begin{proposition}\label{prop:lowerbound2}
For $p\ge 2$, $C(p,2)\ge p-1$.
\end{proposition}
\begin{proof}
Define the vector field $u_k\colon \R^2 \to \R^2$ as 
\begin{equation*}
u_k(x) = Q x \ln(\abs{x})^k \mathds{1}_{\abs{x}<1},
\end{equation*}
where $x=(x_1, x_2) \in \R^2$,
$\abs{x} = \sqrt{x_1^2 + x_2^2}$,
$Q = \mat{0 & 1 \\ -1 & 0}$, $k\in \mathbb N,\ k\ge 1$,
and $\mathds{1}_{\abs{x}<1}$
is the indicator function of the open unit disk
$B_{\R^2}(0,1)$.
Let $r = \abs{x}$. Then
\begin{equation*}
\nabla u_k = \ln(r)^k Q + k \ln(r)^{k-1} r^{-2} M(x),
\end{equation*} where 
\begin{equation*}
M(x) = \mat{x_1x_2 & x_2^2 \\ -x_1^2 & -x_1x_2}.
\end{equation*}
It follows that
\begin{align*}
\mathcal E(u_k)(x)
= k \ln(r)^{k-1} r^{-2} C(x) \mathds{1}_{\abs{x}<1},
\end{align*}
where 
\begin{equation*}
C(x) = \frac{1}{2}(M(x)+M(x)^T)
= \mat{x_1x_2 & (x_2^2-x_1^2)/2
\\ (x_2^2-x_1^2)/2 & -x_1x_2}
\end{equation*}
and that
\begin{equation*}
\mathcal A(u_k)(x) =
\left(\ln(r)^k + \frac{k}{2} \ln(r)^{k-1}\right)
Q \mathds{1}_{\abs{x}<1}
\end{equation*}
Taking the Frobenius norm, since
$|C(x)| = r^2 / \sqrt{2}$ 
it follows that (writing $t=-\ln(r)$)
\begin{align*}
|\mathcal E(u_k)(x)|
&= \abs{k \ln(r)^{k-1} r^{-2}} |C(x)|
= \abs{k} \abs{\ln r}^{k-1} r^{-2} (r^2 / \sqrt{2})
= \frac{\abs{k}}{\sqrt{2}} t^{k-1}\\
|\mathcal A(u_k)(x)|
&= \abs{\ln(r)^{k-1}(\ln r + k/2)} |Q|
= \sqrt{2} \abs{\ln r}^{k-1} \abs{\ln r + k/2}\\
&= \sqrt{2} t^{k-1} \abs{-t + k/2}
= \sqrt{2} t^{k-1} \abs{t - k/2}.
\end{align*}
Taking $L^p$ norms we obtain
\begin{align*}
\norm{\mathcal E(u_k)}_{L^p}^p& = \int_{B_{\R^2}(0,1)}
|\mathcal E(u_k)(x)|^p \diff x
= \int_0^{2\pi} \int_0^1
\left(\frac{\abs{k}}{\sqrt{2}} t^{k-1}\right)^p
r \diff r \diff \theta\\&
= 2\pi \left(\frac{\abs{k}}{\sqrt{2}}\right)^p
\int_0^1 r (-\ln r)^{p(k-1)} \diff r.
\end{align*}
The integral can also be written as
\begin{equation*}
\int_0^\infty e^{-2t} t^{p(k-1)} dt =
\left(\frac{1}{2}\right)^{p(k-1)+1}
\Gamma(p(k-1)+1)
=\mathrel{:}\left(\frac{1}{2}\right)^{p(k-1)+1}I_p.
\end{equation*}
Thus it follows that
\begin{align*}
\norm{\mathcal E(u_k)}_{L^p}^p
&= 2\pi k^p 2^{-p/2} 2^{-(p(k-1)+1)}
\Gamma(p(k-1)+1)\\
&=2\pi k^p 2^{-p/2} 2^{-(p(k-1)+1)} I_p.
\end{align*}
To calculate the norm of the anti-symmetric part,
first observe that 
\begin{align*}
\norm{\mathcal A(u_k)}_{L^p}^p& =
\int_{B_{\R^2}(0,1)} \abs{\mathcal A(u_k)(x)}^p dx
= \int_0^{2\pi} \int_0^1
\left(
\sqrt{2} t^{k-1} \abs{t - k/2}
\right)^p r dr d\theta\\
&= 2\pi (\sqrt{2})^p\int_0^\infty e^{-2t} t^{p(k-1)}
\abs{t - k/2}^p \diff t.
\end{align*}
Define 
\begin{equation*}
J_p\mathrel{:}=\int_0^\infty e^{-t} t^{p(k-1)}
\left|\frac tk - 1\right|^p\diff t
\end{equation*}
so that as a simple calculation via the
change of variable $s=2t$ shows we have
\begin{equation*}
\norm{\mathcal A(u_k)}_{L^p}^p
=2\pi 2^{p/2} J_p 2^{-p(k-1)-1-p}k^p
\end{equation*}
We now want to find a lower bound for the function
$f_k(p) \mathrel{:}=
\norm{\mathcal A(u_k)}_{L^p}/
\norm{\mathcal E(u_k)}_{L^p}$.
Thanks to our calculations we have
\begin{align}
f_k(p)^p = \frac{\norm{\mathcal A(u_k)}_{L^p}^p}
{\norm{\mathcal E(u_k)}_{L^p}^p}
&= \frac{2\pi (\sqrt{2})^p k^p2^{-p} J_p}
{2\pi (k/\sqrt{2})^p I_p}
\nonumber \\
&= \frac{J_p}{I_p} \label{eq:fk_ratio} 
\end{align}
We now prove that $\lim_{k\to \infty} J_p/I_p\ge (p-1)^p$,
which will conclude the proof.
First notice that the ratio admits
an interpretation as an expectation: 
define a probability measure
$d\mu(t)$ on $(0, \infty)$ by
\begin{equation}
 d\mu(t) = \frac{e^{-t} t^{p(k-1)}}{I_p} dt. 
\end{equation}
Then,
\begin{equation}
\frac{J_p}{I_p}
= \int_0^\infty \left|\frac tk - 1\right|^p
\frac{e^{-t} t^{p(k-1)}}{I_p} dt
= \mathbb{E}[\abs{X_k - 1}^p] 
\end{equation}
where $X_k$ is a random variable that is
$(p(k-1)+1,1/k)$ Gamma-distributed.
By Jensen's inequality, $\mathbb{E}[\phi(X)]
\ge \phi(\mathbb{E}[X])$
for convex functions $\phi$, hence
\begin{equation}
\mathbb{E}[|X_k-1|^p] \ge \abs{\mathbb{E}[X_k] - 1}^p 
\end{equation}
Some straightforward calculations
with the $\Gamma$ function prove that
\begin{align*}
\mathbb{E}[X_k]= \frac{(p(k-1)+1)}k,
\end{align*}
so
\begin{align*}
\frac{J_p}{I_p}& \ge \abs{\mathbb{E}[X_k] - 1}^p \\
&= \left|p\frac{k-1}k+\frac1k-1\right|^p
\end{align*}
Substituting this lower bound for $J_p/I_p$ back
into~\eqref{eq:fk_ratio} gives us
\begin{equation*}
 f_k(p) \ge (p-1)\frac{k-1}{k}.
\end{equation*}
Since $C(p,2)\ge f_k(p)$ for any $k$ it follows that
\begin{equation*}
    C(p,2)\ge \liminf_{k\to \infty}f_k(p)=p-1.
\end{equation*}
\end{proof}
\begin{remark}
As we shall see in Section~\ref{sec:calcvar},
for $p\ge 2$, $f_k(p)\le p-1$ as well.
For this particular sequence of radial functions,
this can be proved directly: using the fact that 
(as one easily sees)
\begin{equation*}
    |x-1|^p\le
    (p-1)^p+\frac{(p-1)^{p-1}}{p^{p-2}}(x^p-px^{p-1}),
\end{equation*}
applying this to $x=X_k$ and
taking expected values yields
\begin{equation*}
    \|X_k-1\|_{L^p}^p\le (p-1)^p+
    \frac{(p-1)^{p-1}}{p^{p-2}}
    \mathbb E\left(X_k^p-pX_k^{p-1}\right)=(p-1)^p,
\end{equation*}
where the second step follows by well known 
properties of moments of the Gamma distribution. 
An argument using Chernoff bounds
(see \cite[Section 2.2]{Concentration})
quickly proves $f_k(p)
=(p-1)+\mathcal O\left(1/k \right)$.
\end{remark}
\subsection{Particular cases of Korn's inequality:
Full gradient, divergence-free and low dimensions}
As previously mentioned, there exist many other forms
and variants of the Korn inequality. 
In this section, we shall focus on~\eqref{eq:Kornfull}
and~\eqref{Korntrace}.
We can apply the same method as in the previous sections
to obtain bounds for these as well:
\begin{theorem} For $d\ge 2$ and $p\in (1,\infty)$
\begin{align*}
C_{\mathrm{KN}}(p,d)&\le \begin{cases}
\sqrt{3(p-1)^2+1}&\text{if }p\ge 2\\
\left((\sqrt{3}(p^*-1))^p+1\right)
^{\frac 1p}&\text{if }p\in(1,2)
\end{cases}\\
C_{0}(p,d)&\le \sqrt 3(p^*-1).
\end{align*}
If $p\ge 2$, then
\begin{equation*}
\sqrt{(p-1)^2+1}\le C_{\mathrm{KN}}(p,d)
\end{equation*}
and
\begin{equation*}
p-1\le C_0(p,d).
\end{equation*}
\end{theorem}
\begin{proof}
We first focus on the full Korn inequality.
Concerning lower bounds,
the same argument as for $C(p,d)$ works
mutatis mutandis: the only necessary change in
Proposition~\ref{prop:monotonedim} is the estimate
\begin{equation*}
\|\nabla f_{d+1}\|_{L^p}\ge \|\nabla f_d\|_{L^p},
\end{equation*}
which follows from
\begin{equation*}
\nabla f_{d+1}=
\begin{pmatrix}g(x_{d+1})\nabla f_d(x_1,\dots, x_d)
& g'(x_{d+1})f_d(x_1,\dots, x_d)\\
0&0\end{pmatrix}
\end{equation*}
together with
\begin{align*}
|\nabla f_{d+1}(x)|
&=
\sqrt{|\nabla f_d(x_1,\dots, x_{d})|^2|g(x_{d+1})|^2
+|g'(x_{d+1})|^2|f_d
(x_1,\dots, x_d)|^2}\\&
\ge|\nabla f_d(x_1,\dots, x_{d})||g(x_{d+1})|
\end{align*}
which implies
$\|\nabla f_{d+1}\|_{L^p}
\ge \|\nabla f_d\|_{L^p}\|g\|_{L^p}
=\|\nabla f_d\|_{L^p}$.
Concerning Proposition~\ref{prop:lowerbound2},
the same method works with a slightly
different computation. Concerning the upper bounds, 
for $p\ge 2$ it follows from
\begin{equation*}
(\mathrm{Id}+2P_{\mathrm{Skew}}R\otimes R)
(\mathcal E(u))
=\mathcal E(u)+\mathcal A(u)=\nabla u
\end{equation*}
together with Minkowski's integral inequality and
the pointwise orthogonality of
$P_{\mathrm{Skew}}R\otimes R(\mathcal E(u))$
and $\mathcal E(u)$
\begin{align*}
\|\nabla u\|_{L^p}
&=
\left(
\int_{\R^d} \sqrt{|\mathcal E(u)(x)|^2
+|2P_{\mathrm{Skew}}R\otimes R
(\mathcal E(u))(x)|^2}^p\mathrm{d}x
\right)^{\frac 1p}
\\&\le
\Big\|\big(\|\mathcal E(u)\|_{L^p},\|2P_{\mathrm{Skew}}
R\otimes R\|_{L^p\to L^p}
\|\mathcal E(u)\|_{L^p}\big)\Big\|_{\ell_2}.
\end{align*}
For $p\le 2$, a similar proof works,
using the convexity of $t^{2/p}$ instead.
For $C_0(p,d)$, Lemma~\ref{lemma:linalg2}
proves the result for $d\ge 3$. For $d=2$,
the result follows (with a stronger constant)
from~\cite{Volberg}. Indeed, $\mathrm{Sym}_0(2)
=\mathrm{CO}^-(2)$ and $\mathrm{Skew}_0(2)=
\mathrm{CO}^+(2)$, so the map
$\mathcal E_0(u)\to \mathcal A_0(u)$
corresponds
to (after identifying $\C$ with $\R^2$), 
$T_{BA}$ and the result follows by
the estimates of \cite{Volberg} and duality.
The lower bound follows by 
noticing that in Proposition~\ref{prop:lowerbound2}
the sequence $u_k$ is composed
of divergence-free vector fields. 
\end{proof}
\section{Rank-one convexity,
improved lower bounds and sharpness}
\label{sec:calcvar}
We briefly remind the reader of $C^{\mathrm{rc}}(p,d)$
and its connection to $C(p,d)$.
Let $f_{p,c(p,d)}$ be defined as follows:
\begin{equation*}
f_{p,c(p,d)}(A)\mathrel{:}=
c(p,d)^p|P_{\mathrm{Sym}}(A)|^p
-|P_{\mathrm{Skew}}(A)|^p.
\end{equation*}
It is easy
to see that if $c(p,d)\ge C(p,d)$,
then $f_{p,c(p,d)}^{\mathrm{qc}}(0)=0$,
where $f^{\mathrm{qc}}$
denotes the quasiconvex
envelope of $f$, since
\begin{align*}
f^{\mathrm{qc}}_{p,c(p,d)}(0)
&=\inf_{\phi \in C^{\infty}_0(B_{\R^d}(0,1),\R^d)} 
\fint_{B_{\R^d}(0,1)} f_{p,c(p,d)}(\nabla \phi)
\\&=\inf_{\phi\in C^\infty_0(B_{\R^d}(0,1),\R^d)}
c(p,d)^p\|\mathcal E(\phi)\|_{L^p}^p-
\|\mathcal A (\phi)\|_{L^p}^p.
\end{align*} 
In particular, we can recast our conjecture as
\begin{equation*}
f_{p,p^*-1}^{\mathrm{qc}}(0)=0.
\end{equation*}
A necessary condition for this is, of course, that the
same must be true for the
rank-one convex envelope of the integrand.
Let us denote by $C^{\mathrm{qc}}(p,d)$
the minimal constant
necessary so that $f^{\mathrm{qc}}_{p,c(p,d)}(0)=0$
and similarly $C^{\mathrm{rc}}(p,d)$.
We recall the following result of
Ball and Sivaloganathan
\begin{proposition}[{\cite[Proposition 3.4]{Ball}
    and~\cite[Proposition 2.4]{Sivaloganathan}}]
\label{prop:radial}
Let $u\colon \R^d\to \R^d$ be
a generalised-radial function,
i.e. a function of the form
\begin{equation*}
u(x)=r(\|x\|)x
\end{equation*}
where $r\in C^1([0,\infty), M_d(\R))$
is compactly supported in $[0,1]$. Then, for any given
rank-one convex integrand $f\colon M_d(\R)\to \R$
we have
\begin{equation*}
\fint_{B_{\R^d}(0,1)} f(Du)\diff x\ge f(0).
\end{equation*}
\end{proposition}
Thanks to our radial example
in section~\ref{sec:lowerbound},
it follows that $f_{p,c(p,d)}^{\mathrm{rc}}(0)=0$
is possible
only if $c(p,d)\ge (p-1)$,
which is only interesting when $p\ge 2$.
We now prove that
$C^{\mathrm{rc}}(p,d)\ge (p^*-1)$
for all $p\in (1,\infty)$.
This leads to the sharpened form of the lower bounds
for $C(p,d)$ and $C_{\mathrm{KN}}(p,d)$
stated in Theorem~\ref{thm:Korn}.
\begin{theorem}\label{thm:envelope}
For $p\in(1,\infty)$ and $d\in \N\setminus\{0,1\}$,
\begin{equation}
C^{\mathrm{rc}}(p,d)=p^*-1.
\end{equation}
Similarly, in the trace-free case
we have
\begin{equation*}
C^{\mathrm{rc}}_0(p,d)=p^*-1.
\end{equation*}
\end{theorem}
This theorem is a direct
corollary of~\cite[Theorem 2.2]{Cassese2025}
together with 
the following elementary lemma:
\begin{lemma}
Let $X=\mathrm{Sym}_0(d)$ or 
$X=\mathrm{Sym}(d)$. Then
\begin{equation*}
C(X)\mathrel{:}=
\sup_{R:\mathrm{rank}(R)=1}
\frac{|P_{X^\perp}(R)|}{|P_{X}(R)|}=1.
\end{equation*}
\end{lemma}
\begin{proof}[Proof of Theorem~\ref{thm:envelope}]
This result is a corollary
of~\cite[Theorem 2.2]{Cassese2025}.
Indeed,~\cite[Proposition 3.7]{Cassese2025}
implies $C^{\mathrm{rc}}(p,d)
\le C(\mathrm{Sym}(d))(p^*-1)$,
which together with the above lemma
proves the upper bound,
and~\cite[Theorem 2.2]{Cassese2025}
proves the lower bound. More explicitly, the martingale
$M_n\mathrel{:}= f_n(e_1\otimes e_d+e_d\otimes e_1)
+g_n(e_1\otimes e_d-e_d\otimes e_1)$ proves
(for a suitable choice of $(f_n,g_n)$)
that $C^{\mathrm{rc}}(p,d)\ge p^*-1$.
\end{proof}

A very similar argument,
using~\cite{Ivanisvili2015InequalityTransform}, proves
\begin{proposition}\label{prop:lowerboundfull}
\begin{equation*}
C_{\mathrm{KN}}^{\mathrm{rc}}(p,d)=
c(p)\ge \sqrt{1+(p^*-1)^2},
\end{equation*}
where $c(p)$ is defined as
\begin{equation*}
c(p)\mathrel{:}=\begin{cases}
\sqrt{1+(p^*-1)^2}& \text{if }p
\ge 2 \text{ or } p\le p_0\\
\left(
    1-\frac{2^{1-p}(1-s_0)^{p-1}}{(p-1)(1-s_0)+(2-p)}
\right)^{-\frac1p}&\text{if }p\in(p_0,2)
\end{cases}
\end{equation*}
where $p_0$ ($p_0\approx 1.638$)
and $s_0(p)$ are defined as follows: given
\begin{equation*}
u_p(t)=(p-1)(1+t^2)^{\frac{2-p}2}-p+(1+t)^{2-p}-t(2-p)
\end{equation*}
and $g(x)=(1+x)/(1-x)$, $s_0(p)$ is defined as 
the root of $u_p\circ g=0$, whereas $p_0$ is defined
as the root in $p$ of $u_p(g(-1+2/p))=0$.
The same result holds for
$C^{\mathrm{rc}}_{\mathrm{KN},0}(p,d)$.
\end{proposition}
It is interesting to note
that the above bound is indeed sharper than 
$\sqrt{1+(p^*-1)^2}$ for $p\in (p_0,2)$,
but not by much:
$\|c(p)-\sqrt{1+(p^*-1)^2}\|_\infty\le 0.0005$. 
\begin{figure}[ht!]
\centering 
\includegraphics[width=0.7\textwidth]{
    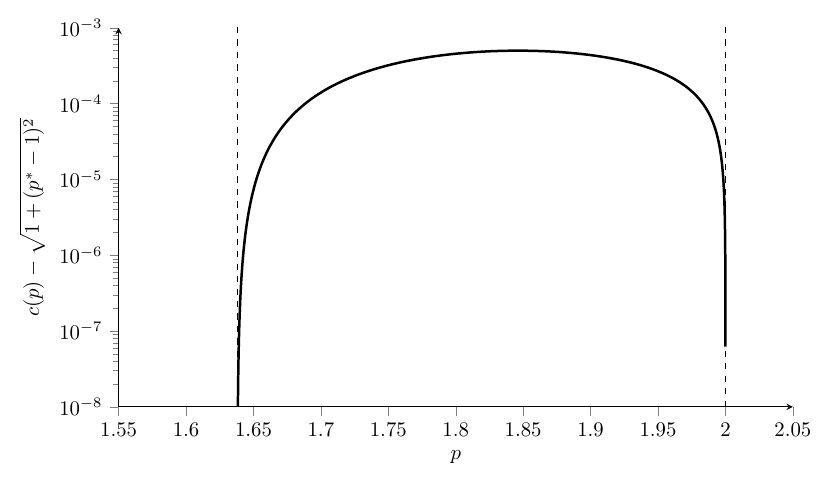
}
\caption{The improvement over the natural bound}
\label{fig:error_plot_pgf}
\end{figure}
The martingale argument we
used in Theorem~\ref{thm:envelope}
can actually be strengthened to give us a partial
description of the rank-one convex
envelope of $f_{p,p^*-1}$:
\begin{proposition}\label{prop:lowerconvex}
Let $\mathcal G_p(x,y)$ be
defined as in \eqref{eq:burkfunct}
\begin{equation*}
\mathcal G_p(x,y)\mathrel{:}=   \begin{cases}
(p^*-1)^p|y|^p-|x|^p&\text{if } |x|\le (p^*-1)|y|\\
p\left(1-\frac1{p^*}\right)^{p-1}
((p^*-1)|y|-|x|)(|x|+|y|)^{p-1}
&\text{otherwise}.
\end{cases}
\end{equation*}
Then
\begin{equation*}
f_{p,p^*-1}^{\mathrm{rc}}(A)
\ge \mathcal G_p(|P_{\mathrm{Skew}(d)}(A)|,
|P_{\mathrm{Sym}(d)}(A)|).
\end{equation*}
Moreover, equality holds whenever
$f_{p,p^*-1}(A)\ge 0$. In dimension $2$,
equality also holds
provided $\mathrm{Tr}(A)= 0$.  
\end{proposition}
\begin{proof}
Thanks to Burkholder's construction, it follows that
given two $M_d(\R)$-valued martingales
$X_n,Y_n$ such that
$X_n\preccurlyeq Y_n$, then
\begin{equation*}
    (p^*-1)^p\|Y_n\|_{L^p}^p-\|X_n\|_{L^p}^p\ge
    \mathcal G_p(|X_0|,|Y_0|).
\end{equation*}
Using the same connection between laminates
and martingales as in the previous
proof, it follows that 
$f_{p,p^*-1}^{\mathrm{rc}}(A)
\ge \mathcal G_p(|P_{\mathrm{Skew}(d)}(A)|,
|P_{\mathrm{Sym}(d)}(A)|)$.
Since whenever $f_{p,p^*-1}(A)\ge 0$, we have 
\begin{equation*}
f_{p,p^*-1}(A)\ge f_{p,p^*-1}^{\mathrm{rc}}(A)
\ge \mathcal G_p(|P_{\mathrm{Skew}(d)}(A)|,
|P_{\mathrm{Sym}(d)}(A)|)
=f_{p,p^*-1}(A),
\end{equation*}
the first claim follows.
To prove equality for $d=2$ when $\mathrm{Tr}(A)=0$,
first note that by invariance under conjugation by 
$\mathrm{O}(2)$ of the integrand $f_{p,p^*-1}$
we may assume
$A$ to be of the form 
\[A=\begin{pmatrix}0&a+b\\a-b&0\end{pmatrix}.\]
Then the martingale
\begin{equation*}
    M_n\mathrel{:}=
    \frac{f_n}{\sqrt 2}(e_1\otimes e_2+e_2\otimes e_1)+
    \frac{g_n}{\sqrt 2}(e_1\otimes e_2-e_2\otimes e_1),
\end{equation*}
where $f_0=\sqrt 2a,$ $g_0=\sqrt 2b$
and $|dg_n|=|df_n|$, has 
rank-one increments and starts at $A$, hence
its law describes a laminate with barycentre $A$. Since
\[
\mathbb E(f_{p,p^*-1}(M_n))=
(p^*-1)^p\|f_n\|_{L^p}^p-
\|g_n\|_{L^p}^p,
\]
by \eqref{eq:maxdyad} it follows that
\begin{equation*}
    \inf_{(f_n,g_n)}\mathbb E(f_{p,p^*-1}(M_n))
    =\mathcal G_p(|g_0|, |f_0|)=
    \mathcal G_p(|P_{\mathrm{Skew}(2)}(A)|,
    |P_{\mathrm{Sym}(2)}(A)|)
\end{equation*}
and the result is proved.
\end{proof}
\begin{remark}
Proposition~\ref{prop:lowerconvex} can be easily
adapted into one for projection onto 
arbitrary subspaces $X$ of $M_d(\R)$. In the
notation of \cite{Cassese2025}
\[f_{X,p,C(X)(p^*-1)}^{\mathrm{rc}}(A)
\ge \mathcal G_p(|P_{X^\perp}(A)|, C(X)|P_X(A)|)
.\]
In particular, for $X=\mathrm{CO}^+(2)$, 
this provides an alternative
proof of the exact rank-one convex envelope.
The reason this argument is sharp in that case
is that the condition
$|P_{\text{CO}^+(2)}(A)|=|P_{\text{CO}^-(2)}(A)|$
is equivalent to $\det(A)=0$, so martingale
transforms in this setting
coincide exactly with martingales
having rank-one increments, which is
why one gets sharpness.
In Korn's case, we don't expect this bound to be
sharp in arbitrary dimension, since
the set of pairs of symmetric--skew-symmetric 
martingales $(A_n,B_n)$ with 
$\mathrm{rank}(dA_n+dB_n)\le 1$ seems 
rather different from that of
differentially subordinate
pairs.
To calculate the exact envelope,
a Bellman function argument
is probably needed, though the geometry
of the constraint 
$\mathrm{rank}(dA_n+dB_n)\le 1$ frustrates
naïve attempts.
\end{remark}
\begin{corollary}[Radial bounds in
Theorem~\ref{thm:Korn}]
Let $u\in W^{1,p}_0(B_{\R^d}(0,1),\R^d)$ be a
generalised-radial map, i.e.
\begin{equation*}
u(x)=r(\|x\|)x
\end{equation*}
for some function $r\colon\R^+\to M_d(\R)$. Then
\begin{equation*}
\|\mathcal A(u)\|_{L^p}
\le (p^*-1)\|\mathcal E(u)\|_{L^p}.
\end{equation*}
Similarly we have
\begin{equation*}
\|\nabla u\|_{L^p}\le c(p)\|\mathcal E(u)\|_{L^p}.
\end{equation*}
Moreover, for $p\ge 2$ both inequalities are sharp
in dimension $2$.
\footnote{It is unclear whether these inequalities
are sharp in higher dimensions as well,
but preliminary numerical evidence suggests
this is unlikely to be the case.}
The exact same results hold for the trace-free
versions of the inequality.
\end{corollary}
\begin{proof}
The result quickly follows
from Proposition~\ref{prop:radial}
and a limiting argument to deal with
general $u\in W^{1,p}_0$. 
Sharpness follows from the examples constructed
in Proposition~\ref{prop:lowerbound2}, which are radial.
\end{proof}
For $p\in (1,2)$, the 
behaviour of the constant
is, somewhat surprisingly,
much tamer. Indeed, 
we have a Korn inequality
even for $L^1$:
\begin{proposition}
Let $u\in W^{1,1}(\R^d)$ be generalised-radial.
Then for $d\ge 3$
\begin{equation*}
\|\mathcal A(u)\|_{L^1}\le
\left(1+\mathcal O\left(\frac1{\sqrt d}\right)\right)
\|\mathcal E(u)\|_{L^1}.
\end{equation*}
Similarly
\begin{equation*}
\|\mathcal A_0(u)\|_{L^1}
\le \left(1+\mathcal O
\left(\frac1{\sqrt d}\right)\right)
\|\mathcal E_0(u)\|_{L^1}.
\end{equation*}
Both constants are asymptotically sharp,
meaning that as $d\to \infty$ the constants
tend to $1$.
\end{proposition}
Before proving the result, we need an auxiliary lemma:
\begin{lemma}\label{lemma:integralmatrix}
Let $A\in M_d(\R)$. Then
\begin{equation*}
\fint_{\mathbb{S}^{d-1}}|A\theta|\mathrm d\theta
\ge c_d |A|,
\end{equation*}
where $c_d$ is defined as
\begin{equation*}
c_d\mathrel{:}=\fint_{\mathbb{S}^{d-1}}
|\theta_1|\mathrm{d}\theta
=\frac{\Gamma(\frac d2)}
{\sqrt\pi\,\Gamma(\frac{d+1}{2})}.
\end{equation*}
\end{lemma}
\begin{proof}
We can assume without loss of generality that $|A|=1$.
Using the singular value decomposition
we can assume that $A$ is diagonal,
$A=\mathrm{diag}(\lambda_1,\dots, \lambda_d)$ so
\begin{equation*}
\fint_{\mathbb{S}^{d-1}}|A\theta|\mathrm{d}\theta
=\fint \sqrt{\sum \lambda_i^2\theta_i^2}\mathrm{d}\theta
=\colon f(\lambda_1^2,\dots, \lambda_d^2).
\end{equation*}
The function $f(s_1,\dots, s_d)$ is concave
and defined on
the simplex $\{\sum s_i=1, s_i\ge 0\}$,
so its min is attained at the vertices of the simplex,
and the result follows.
\end{proof}
\begin{proof}
We assume without loss of generality
$u=M(\|x\|)x$ with $M(r)$
smooth and compactly supported in $[0,1]$.
As a straightforward calculation shows, we have
\begin{equation*}
\nabla u= M(r)+(rM'(r) \theta)\otimes \theta,
\end{equation*}
where $r=\|x\|, \theta=x/\|x\|$. It follows that
\begin{equation*}
    |\mathcal A(u)|\le |\mathcal E(u)|+\sqrt 2 |M(r)|.
\end{equation*}
It thus suffices to prove that
$\|M(r)\|_{L^1}\le
\mathcal O\left(1/\sqrt d\right)
\|\mathcal E(u)\|_{L^1}$.
Since
\begin{equation*}
    \frac{1}{\omega_{d-1}}\|M\|_{L^1}=
    \int_0^\infty r^{d-1}|M(r)|\mathrm{d}r\le
    \frac1d \int_0^\infty r^{d}|M'(r)|\mathrm{d}r
\end{equation*}
where $\omega_{d-1}$ is the surface area
of $\mathbb S^{d-1}$,
we focus our attention on estimating $M'(r)$.
Define the radial averages
$e(r), d(r)$ respectively as
\begin{align*}
e(r)
&\mathrel{:}=
\fint_{\mathbb S^{d-1}}
|\mathcal E(u)(r\theta)|\mathrm{d}\theta\\ 
d(r)
&\mathrel{:}=
\fint_{\mathbb S^{d-1}}
|\mathrm{Sym}(rM'(r)\theta\otimes \theta)|
\mathrm{d}\theta.
\end{align*}
As one easily checks, 
\begin{equation*}
d(r)\ge \frac{r}{\sqrt{2}}
\fint_{\mathbb S^{d-1}} |M'(r)\theta|\mathrm{d}\theta.
\end{equation*}
Thanks to Lemma~\ref{lemma:integralmatrix},
we can bound the right-hand side from below, obtaining
\begin{equation*}
d(r)\ge \frac{c_d}{\sqrt{2}}r|M'(r)|.
\end{equation*}
Moreover $d(r)\le e(r)+|\mathrm{Sym}(M(r))|$, hence
\begin{equation*}
    |M'(r)|\le \frac{\sqrt 2}{c_d r}
    \left(e(r)+|\mathrm{Sym}(M(r))|\right).
\end{equation*}
Integrating this against $r^d\mathrm{d}r$ we obtain
\begin{equation*}
    \int_0^\infty |M'(r)|r^d\mathrm{d}r
    \le \frac{\sqrt 2}{c_d \omega_{d-1}}
    (\|\mathcal E(u)\|_{L^1}+\|\mathrm{Sym}(M)\|_{L^1}).
\end{equation*}
On the other hand, since $\fint_{\mathbb{S}^{d-1}}
\theta\otimes \theta=1/d \mathrm{Id}$, we have
\begin{equation*}
    \fint_{\mathbb S^{d-1}} \mathcal E(u)(r\theta)
    \mathrm d \theta
    =\mathrm{Sym} M(r)+\frac rd\mathrm{Sym} M'(r)
\end{equation*}
and so we can estimate
$|\mathrm{Sym}(M(r))|$ from above by 
\begin{equation*}
    |\mathrm{Sym}(M(r))|\le e(r)+\frac{r}{d}|M'(r)|.
\end{equation*}
Integrating this proves
\begin{equation*}
    \|\mathrm{Sym}(M(r))\|_{L^1}
    \le \|\mathcal E(u)\|_{L^1}+
    \frac{\omega_{d-1}}d \int_0^\infty
    |M'(r)|r^d\mathrm{d}r.
\end{equation*}
Combining the two inequalities, it follows that
\begin{equation*}
\left(1-\frac{\sqrt 2}
{c_d d}\right)\int_0^\infty |M'(r)|r^{d}\mathrm{d}r
\le \frac{2\sqrt 2}{c_d\omega_{d-1}}
\|\mathcal E(u)\|_{L^1}. 
\end{equation*}
Therefore if $c_d d>\sqrt 2$
(which is always true provided $d>2$) we have
\begin{align*}
    \|M\|_{L^1}&\le \frac{\omega_{d-1}}d
    \frac{1}{1-\frac{\sqrt 2}{c_d d}}
    \frac{2\sqrt 2}{c_d\omega_{d-1}}
    \|\mathcal E(u)\|_{L^1}
    \\&=\frac{2\sqrt 2}
    {c_dd-\sqrt 2}\|\mathcal E(u)\|_{L^1}\\
    &=\mathcal O\left(\frac{1}{\sqrt d}\right)
    \|\mathcal E(u)\|_{L^1}.
\end{align*}
This completes the proof of the estimate.
The proof for the trace-free
version of Korn follows in exactly the same way.
To prove asymptotic sharpness, it suffices to consider 
explicit witnesses: for example, in even dimensions,
$u(x)=\chi_{|x|\le 1}(1-|x|^2)
\begin{pmatrix}
0&\mathrm{Id}\\ -\mathrm{Id}&0
\end{pmatrix}x$
witnesses the asymptotic sharpness of the estimate.
\end{proof}
\begin{remark}
A similar behaviour, though with the roles
of the intervals $(1,2)$ and $(2,\infty)$
switched, is known for the Beurling-Ahlfors transform
(see~\cite{Banuelos2013},~\cite
{Strzelecki2017TheFunctions}).
\end{remark}
These considerations suggest the following
strengthening of our conjecture:
\begin{conjecture}
$-\mathcal G_p(|P_{\mathrm{Skew}(d)}(A)|,
|P_{\mathrm{Sym}(d)}(A)|)$
is quasiconvex at some matrix $A_0$.
\footnote{Note that quasiconvexity at $A_0$ implies,
by homogeneity of the function, quasiconvexity at $0$.}
\end{conjecture}
\begin{theorem}\label{thm:Korn2d-bound}
    \begin{equation} \label{eq:Korn2d-bound}
        p^*-1\le C(p,2)\le 1.158(p^*-1).
    \end{equation}
\end{theorem}
\begin{proof}
    The proof will take advantage of the complex structure on 
    $\R^2$. More precisely,
identify a function $u\colon \R^2\to \R^2$
with $g\colon \C\to \C$
defined as
\[
g(z)=u_1(x,y)+iu_2(x,y).
\]
Then it follows that
\begin{align*}
    |\bar \partial g(z)|&=
    \sqrt 2|\mathcal E_0(u)(x,y)|,\\
    |\Im(\partial g)(z)|&=
    \sqrt 2|\mathcal A(u)(x,y)|.
\end{align*}
For compactly supported $u$, the Beurling--Ahlfors
transform satisfies
\[
    T_{BA}(\bar\partial g)=\partial g.
\]
The operator relevant to Korn's inequality is therefore the
real-linear operator\footnote{Not to be confused with the complex-linear operator $\Im(T_{BA})$}
\[
    \mathcal I f\mathrel{:=}\Im(T_{BA}f),
    \qquad
    \mathcal I\colon L^p(\C;\C)\longrightarrow L^p(\C;\R).
\]
The identities above imply
\begin{equation}\label{eq:Korn-via-BA}
    C(p,2)\le \|\mathcal I\|_{L^p(\C;\C)\to L^p(\C;\R)}.
\end{equation}
To estimate this norm, write $T_{BA}=A+iB$ on real-valued
functions. If $f=f_1+if_2$, then
\[
    \mathcal I f=Bf_1+Af_2.
\]
Since $A$ and $B$ are self-adjoint, the adjoint of $\mathcal I$,
after identifying $\C$ with $\R^2$, satisfies
\[
    \mathcal I^*h=(Bh,Ah),
    \qquad |\mathcal I^*h|=|T_{BA}h|
\]
for real-valued $h$. Consequently,
\begin{align*}
\|\mathcal I\|_{L^p(\C;\C)\to L^p(\C;\R)}
&=\|\mathcal I^*\|_{L^{p'}(\C;\R)\to L^{p'}(\C;\C)}\\
&=\|T_{BA}\|_{L^{p'}(\C;\R)\to L^{p'}(\C;\C)}.
\end{align*}
Since the Beurling-Ahlfors transform, when restricted to real-valued functions, satisfies
(see~\cite[p.~46]{Banuelos2026Cotlar})
\[
    \|T_{BA}\|_{L^q(\C;\R)\to L^q(\C;\C)}
    \le 1.158(q^*-1)
\]
for $q\in(1,\infty)$,
taking $q=p'$ in this
estimate yields
\[
    C(p,2)\le 1.158(p^*-1).
\]
Together with the lower bound from Theorem~\ref{thm:envelope},
the proof is complete.
\end{proof}
\begin{remark}\label{rmk:iwaniecreal}
It follows from the proof of the above theorem that 
\[
p^*-1=C(p,2)
\]
would be a consequence of the Iwaniec
conjecture for real-valued functions.
\end{remark}
\section{Korn in other function spaces}
As an application of the previous 
results (in particular of the
differential subordination obtained),
we briefly outline some
other dimension-free estimates that can be obtained
for Korn's inequality in other spaces:
Sobolev-Orlicz spaces and weighted Sobolev spaces.
\subsection{Orlicz spaces}
\label{sec:orlicz}
Let us briefly recall the necessary notions
concerning Orlicz spaces. Given a Young function $\Phi$,
it is well known (see~\cite{Breit2011}) that
Korn's inequality holds in a Sobolev-Orlicz
space $\dot{W}^1L^{\Phi}$
if and only if
$\Phi$ satisfies both $\nabla_2$ and $\Delta_2$.
Before stating our result, let us recall
some results and notation. 
Given a Young function $\Phi$ which we assume $C^1$,
we define its Simonenko indices as follows:
\begin{align*}
    i_\Phi=
    \inf_{x\in (0,\infty)} \frac{x\Phi'(x)}{\Phi(x)}
\end{align*}
and its upper index as
    \begin{align*}
    s_\Phi=
    \sup_{x\in (0,\infty)} \frac{x\Phi'(x)}{\Phi(x)}
\end{align*}
Recall that $1\le i_\Phi\le s_\Phi\le +\infty$
and that $1<i_\Phi\le s_\Phi<+\infty$ is equivalent to
$\nabla_2$ and $\Delta_2$ being true for $\Phi$. In this
case, the constant in the Korn inequality is 
dimension-free and, indeed, its dependence
on $\Phi$ is in some sense the best one could hope for.
This essentially follows from Theorem~\ref{thm:Korn} and
interpolation theory for Orlicz spaces.
\begin{theorem}\label{thm:Orlicz}
Let $\Phi$ be a Young function satisfying 
both the $\Delta_2$ and the $\nabla_2$
conditions and let $p=i_{\Phi}$
and $q=s_{\Phi}$ be its Simonenko indices. Then 
\begin{equation*}
    \left\|\mathcal A(u)\right\|_{L^{\Phi}}
    \le \sqrt{3}K\max(p^*-1,q^*-1)
    \left\|\mathcal E(u)\right\|_{L^{\Phi}},
\end{equation*}
where the constant $K=K(p,q)$ satisfies 
\begin{equation*}
K\le 2^{\frac{1}{pq'}
+\min\left(\frac1p,\frac1{q'}\right)}<4
\end{equation*}
In particular, if $p,q\in (1,2]$ or $p,q\in [2,\infty)$
then we get
\begin{equation*}
    C(\Phi, \R^d)\le 2\sqrt 3\max(p^*-1,q^*-1).
\end{equation*}
\end{theorem}
To prove the result, we need a suitable 
version of Riesz-Thorin:
\begin{theorem}
    [{\cite[Theorem 5.1 and Remark 5.2]
    {RieszThorinOrlicz}}]
Take $p,q\in (1,\infty)$ and let
$\Phi$ be a $C^1$ $N$-function
such that $p<i_\Phi<s_\Phi<q$.
Then, given a linear operator $T$ bounded
on $L^p$ and $L^q$
respectively, $T$ is also bounded on $L^{\Phi}$
(with both the Luxemburg and the Orlicz norm) 
and its norm is bounded as follows:
\begin{equation*}
\|T\|_{L^\Phi}\le K\max(\|T\|_{L^p}, \|T\|_{L^q}),
\end{equation*}
where the constant $K=K(p,q)$ satisfies 
\begin{equation*}
K\le 2^{\frac{1}{pq'}+
\min\left(\frac1p,\frac1{q'}\right)}<4.
\end{equation*}
\end{theorem}
Theorem~\ref{thm:Orlicz} is now a quick corollary
of Theorem~\ref{thm:Korn}
and the interpolation
theorem with a limiting argument.
As a quick application of this result,
we derive a sharpened form of the Korn inequality
obtained by Acerbi and
Mingione~\cite{AcerbiMingione}, namely
\begin{corollary}
Let $G_{\lambda,p}(t)=(1+\lambda t)^{p-2}t^2$
with $p\in (1,2]$. Then
\begin{equation*}
\|\mathcal A(u)\|_{L^{G_{\lambda,p}}}
\le 2\sqrt 3(p^*-1)
\|\mathcal E(u)\|_{L^{G_{\lambda,p}}}
\end{equation*}
\end{corollary}
\begin{remark}
It is possible to prove the estimate on Orlicz spaces
by directly mimicking the process used for $L^p$:
namely, provided one has a version of
Burkholder's inequality for
Orlicz spaces the same proof
carries through mutatis mutandis
and this allows one to obtain sharper bounds on $K$.
This will be done in follow-up work,
but let us note that some progress
in this direction (that is, extending
Burkholder's inequality
to Orlicz spaces) has already been made (under some
additional assumptions on $\Phi$)
in~\cite{OsekowskiOrlicz}.
The results of that paper, together with
Theorem~\ref{thm:Korn}, imply
e.g. that if we assume $\Phi$ is increasing, convex
and $C^2$ with $\Phi(0)=\Phi'(0)=0$,
$\lim_{x\to \infty}\Phi''(x)=0$ and 
\begin{equation*}
9(C-1)\Phi'(x)\le x\Phi''(x)< \Phi'(x)
\end{equation*}
for some $C>1$ then 
\begin{equation*}
C(\Phi,\R^d)\le \frac1{3(C-1)},
\end{equation*}
and a similar result is available
in the dual case $\Phi''(0)=0$.
In particular, this allows us to get
an improved estimate
in the previous corollary, as a quick
calculation shows that we can take
$C-1=(p-1)/9$, hence $C(G_{\lambda,p}, d)\le 3(p^*-1)$.
Note however that an improved estimate on $C(p,d)$
is likely to improve
further on this bound.
\end{remark}
The above remark can be made more precise.
Let $C(\Phi, \R^d)$
be the best constant such that,
for all $u\in C^\infty_c(\R^d, \R^d)$
we have
\begin{equation*}
\int_{\R^d}\Phi\Bigl(
C(\Phi,d) |P_{\mathrm{Sym}(d)}(\nabla u)|
\Bigr)
\mathrm{d}x
\ge \int_{\R^d} \Phi\Bigl(
|P_{\mathrm{Skew}(d)}(\nabla u)|
\Bigr)\mathrm{d}x.
\end{equation*}
Equivalently, this can be stated
in terms of the quasiconvexity
at $0$ of
\[
f_C(A):=\Phi(C(\Phi,d) |P_{\mathrm{Sym}(d)}(A)|)-
\Phi(|P_{\mathrm{Skew}(d)}(A)|).
\]
Then the arguments in Theorems~\ref{thm:envelope}
and~\ref{thm:Korn} easily adapt to prove 
\begin{align*}
\sqrt 3 C_{\mathrm{Diff}}(\Phi, H(d))
&\ge C^{\mathrm{qc}}(\Phi,d),\\ 
C_{\mathrm{Diff}}(\Phi, H(d))
&\ge C^{\mathrm{rc}}(\Phi,d)\ge 
C_{\mathrm{Sc}}(\Phi, \mathbb{R})
\end{align*}
where $C_{\mathrm{Diff}}(\Phi, H)$
is the best constant for differential
subordination bounds in $L^\Phi$
for $H$-valued martingales,
$C_{\mathrm{Sc}}(\Phi)$ the
best constant for $\pm1$ dyadic
martingale transform bounds and 
$H(d)$ is any $d(d+1)/2$-dimensional Hilbert space.
Since given an $H$-valued
continuous martingale $X_t$
there exists a continuous martingale $\tilde X_t$ taking values in $\R^2$
such that almost surely we 
have
\begin{align*}
    |X_t|&=|\tilde X_t|\\
    \langle X_t\rangle&=\langle \tilde X_t\rangle
\end{align*}
as proved in \cite{KallenbergSztencel},
it follows that if $\mathrm{dim}(H)\ge 2$,
$C_{\mathrm{Diff}}(\Phi, H)=
C_{\mathrm{Diff}}(\Phi, \R^2)$, so the previous bounds become dimension-free.
It is natural to expect that
$C_{\mathrm{Diff}}(\Phi, H)=
C_{\mathrm{Diff}}(\Phi, \R)=C_{\mathrm{Sc}}(\Phi, \R)$
for any
separable Hilbert space $H$, but this is not known at the moment.
However, these bounds are already quite
useful in and of themselves:
since $C_{\mathrm{Sc}}(\Phi, \R)<+\infty$
if and only if $\Phi$ satisfies the
$\Delta_2$ and $\nabla_2$ conditions, 
the above estimate implies a quick proof of a weak
form of~\cite{Breit2011}:
\begin{proposition}
Given a bounded open set $\Omega\subset \mathbb R^d$,
the inequality 
\begin{equation}\label{eq:KornOrlicz}
C_1+\int_{\Omega}\Phi\Bigl(
C_2 |P_{\mathrm{Sym}(d)}(\nabla u)|\Bigr)
\mathrm{d}x
\ge \int_{\Omega} \Phi\Bigl(
|P_{\mathrm{Skew}(d)}(\nabla u)|\Bigr)
\mathrm{d}x
\end{equation}
holds for some $C_1,C_2$
for all $u\in C^\infty_0(\Omega, \R^d)$ if and only if 
$\Phi$ satisfies both $\nabla_2(\infty)$
and $\Delta_2(\infty)$
and the same holds for the trace-free
version of Korn's inequality.
\end{proposition}
\begin{proof}
We can without loss of generality assume $\Phi$
to be real-valued, for otherwise a smooth 
approximation of $u_1(x)$ defined as
in Proposition~\ref{prop:lowerbound2}
(and similar functions 
in higher dimensions)
proves the inequality cannot hold.
Let $f_C$ denote the integrand associated
to~\eqref{eq:KornOrlicz},
that is
\[
f_C(A)\mathrel{:}=
\Phi(C |P_{\mathrm{Sym}(d)}(A)|)
-\Phi(|P_{\mathrm{Skew}(d)}(A)|).
\]
Then a necessary condition for~\eqref{eq:KornOrlicz}
to hold is $f_C^{\mathrm{rc}}(0)>-\infty$.
By considering the martingale
$f_n (e_1\otimes e_d+e_d\otimes e_1)
+g_n (e_1\otimes e_d-e_d\otimes e_1)$,
this can only happen if
\begin{equation*}
\mathbb E(\Phi(C |f_n|))\ge \mathbb E(\Phi(|g_n|)).
\end{equation*}
This only happens if the martingale $\pm1$ transform
is bounded,
which by \cite[Theorem 2.c.6]{Lindenstrauss1996}
can only happen
if $\Phi$ satisfies both
$\Delta_2(\infty)$ and $\nabla_2(\infty)$.
For sufficiency,
standard interpolation arguments
such as the ones outlined previously in the section suffice,
since by modifying
$\Phi$ on a finite interval $[0,a]$ we can assume without loss
of generality that
it satisfies $\nabla_2$ and $\Delta_2$ globally.
\end{proof}
This technique can be pushed further: 
the same argument proves a slightly weaker form of one of
the main results of 
\cite{BreitCianchiDiening2017}:
\begin{proposition}\label{prop:KornOrliczcondition}
Given a bounded open set $\Omega\subset \mathbb R^d$,
the inequality
\begin{equation}
\label{eq:KornOrlicz2}
C_1+\int_{\Omega}\Phi\Bigl(
C_2 |P_{\mathrm{Sym}(d)}(\nabla u)|\Bigr)
\mathrm{d}x
\ge \int_{\Omega} \Psi\Bigl(
|P_{\mathrm{Skew}(d)}(\nabla u)|\Bigr)
\mathrm{d}x
\end{equation}
holds for some $C_1,C_2$ for all $u\in C^\infty_0(\Omega, \R^d)$
if and only if the Calderón operator $S=C+C^*$, where 
\begin{equation*}
Cf(t)\mathrel{:}=\frac1t\int_0^t f(s)\mathrm{d}s
\end{equation*}
and $C^*$ is its formal adjoint, is bounded between
$L^{\Phi}([0,1])$
and $L^{\Psi}([0,1])$. The same holds for Korn's
trace-free inequality.
\end{proposition}
\begin{proof}
We can without loss of generality assume $\Psi$
to be real-valued, for otherwise a smooth 
approximation of $u_1(x)$ defined as
in Proposition~\ref{prop:lowerbound2}
(and similar functions 
in higher dimensions)
proves the inequality cannot hold.

Let $f_C$ denote the integrand associated
to~\eqref{eq:KornOrlicz2},
that is
\[
f_C(A)\mathrel{:}=
\Phi(C |P_{\mathrm{Sym}(d)}(A)|)
-\Psi(|P_{\mathrm{Skew}(d)}(A)|).
\]
Then a necessary condition for~\eqref{eq:KornOrlicz2} to hold is
$f_C^{\mathrm{rc}}(0)>-\infty$.
By considering the martingale
$f_n (e_1\otimes e_d+e_d\otimes e_1)+
g_n (e_1\otimes e_d-e_d\otimes e_1)$,
this can only happen if
\begin{equation*}
\mathbb E(\Phi(C |f_n|))\ge \mathbb E(\Psi(|g_n|)).
\end{equation*}
This only happens if the martingale $\pm1$ transform is bounded, 
which by~\cite{Astashkin2023} is equivalent to the
boundedness of $S$. The sufficiency of this condition follows
from standard interpolation arguments:
since $R\otimes R$ is self-adjoint
and bounded on $L^2$, it follows that
(see \cite[Appendix]{Calderon1966}) 
\[(R\otimes R f)^*(x)\lesssim S f^*(x),\]
where $f^*$ is the decreasing right-continuous
rearrangement of $f$.
The result follows, since rearrangement
preserves Orlicz modulars and if $S$
is bounded between $L^{\Phi}$ and $L^{\Psi}$
it automatically satisfies the required modular inequality.
\end{proof}
\begin{remark}
The above argument can be extended
to prove equivalence with the norm
inequality, and to more general function spaces and
differential operators, though
this requires some new arguments.
This will be explored in future work, 
in part jointly with Bernd Kirchheim, Jan 
Kristensen and Wen Qi Zhang.
\end{remark}
\subsection{Korn in weighted Sobolev spaces}
Recently, Domelevo, Petermichl and
Skreb (\cite{Petermichl}) have 
extended Burkholder's differential subordination results
to weighted $L^p$ spaces.
This result, combined
with the space-time martingale
representation of $R\otimes R$,
yields a weighted Korn inequality.\par 
It is well known that, given a Calderón-Zygmund operator
and a Muckenhoupt $A_p$ weight $w$, it acts continuously
on weighted $L^p(w)$ spaces,
with the following bound:
\begin{equation*}
    \|T\|_{L^p(w)\to L^p(w)}
    \le C_{d,p}(T)[w]_{A_p}^{\max(1,\frac1{p-1})}.
\end{equation*}
This result, known as the $A_p$ conjecture,
was proved by Hytönen in~\cite{Hytonen2}. In our context, it is 
natural to conjecture $\|R\otimes R\|_{L^p(w)\to L^p(w)}
\lesssim_d (p^*-1)[w]_{A_p}^{\max(1,1/(p-1))}
$ and perhaps even that the inequality
should hold dimension-free.
We will prove this is indeed the case as a straightforward
application of the weighted $L^p$ estimates
proved in~\cite{Petermichl},
provided one replaces $[\cdot]_{A_p}$ by
an equivalent characteristic.\par
We define the following notion of weights,
which, given our
reliance on heat methods,
is much better suited to our approach
than Muckenhoupt's original definition:
\begin{definition}[Heat weight]
Given a function $f\in L^1_{\mathrm{loc}}(\R^d), $ 
we say $f\in  A_p^{\mathrm{Heat}}$ (the $p$-class of heat 
weights) if $f>0$ almost everywhere and 
\begin{equation*}
    [f]_{A^{\text{\upshape Heat}}_p}
    \mathrel{:}=
    \sup_{\R^{d+1}_+}|
    U f(x,t) U g(x,t)^{p-1}|<\infty,
\end{equation*}
where $U f$ denotes the heat extension of $f$
and $U g$ that of $g(x)\mathrel{:}=f(x)^{-1/(p-
1)}$.
\end{definition}
Considering this class of weights does not entail any loss of 
generality, thanks to the following well-known result
\begin{proposition}
    [{\cite[Theorem 3.4]{Fefferman}, 
    \cite[Theorem 3.2]
    {Petermichl2002HEATINGQUASIREGULAR}}]
For $d\ge 1$, $p\in (1,\infty)$
\begin{equation*}
A^{\text{\upshape Heat}}_p= A_p
\end{equation*}
\end{proposition}
\begin{remark}
It is important to note that,
while the two spaces $A_p$
and $A^{\text{\upshape{Heat}}}_p$ coincide, this 
does not imply that the Muckenhoupt and the heat
characteristic are linearly equivalent;
Instead, what we obtain is
\begin{equation*}
[w]_{A_p}
\lesssim
[w]_{A^{\text{\upshape Heat}}_p}
\lesssim D(w) [w]_{A_p},
\end{equation*}
where $D(w)$ is the doubling constant
associated to the measure $w\diff x$ (which
can be bounded from above by
$D(w)\lesssim 2^{dp}[w]_{A_p}$).
We also remark that the previous
result fails for the Poisson kernel because 
of its slow decay, see section 3
of~\cite{Fefferman} for more details.
This is, in particular, the reason
why we choose to use the space-time
Brownian motion representation of $R_iR_j$
instead of the background-diffusion one.
\end{remark}
We will now use a weighted version of
Burkholder's method
(\cite[corollary 1]{Petermichl}) to obtain
\begin{theorem}
Given a weight $w\in A_{p}^{\text{\upshape Heat}}$
we have
\begin{equation*}
C(p,w\diff x, \mathbb R^d)
\le K (p^*-1) 
[w]_{A_{p}^{\text{\upshape Heat}}}
^{\max(1,\frac{1}{p-1})},
\end{equation*}
where $K$ is an absolute constant and $C(p,w\diff x, \R^d)$ is the best constant
for the associated Korn inequality.
\end{theorem}
Let us now recall the weighted
martingale estimate that
underlies the proof of the $A_p$-version
of our estimates:
\begin{theorem}[{\cite[Corollary 1]{Petermichl}}]
\label{thm:burkholderweighted}
Let $(X_t),(Y_t)$ be Hilbert-space-valued martingales and
assume that $X \preccurlyeq Y$.
Then, given a martingale $w\in A_p$
with $p\in (1,\infty)$ we have
\begin{equation*}
\vvvert{X}_{L^p(w)}\lesssim (p^*-1)
[w]_{A_p^{\text{\upshape mart}}}^{\max\left(\frac{1}{p-1
},1\right)} \vvvert{Y}_{L^p(w)}.
\end{equation*}
\end{theorem}
The last ingredient we shall need is the connection between
heat weights and martingale weights,
which is a straightforward consequence
of Itô's formula.
\begin{lemma}
[{\cite[Section 4]{banuelos2017weightedsquarefunctionestimates}}]
Given a heat weight
$w$, define $W$ as $W^{T,n}_t=U w(Y_t^{T,n})$,
where $Y_t^{T,n}$ is the space-time Brownian motion with
uniform initial distribution on $[-n,n]^d$ and $Uw$
the heat extension of $w$. Then
\[
[w]_{A_p^{\mathrm{Heat}}}\ge [W^{T,n}]_{A_p^{\mathrm{mart}}}
\]
\end{lemma}

\begin{remark}
The reason we deal with a restricted version of $W$
rather than with $W$ itself is so as to sidestep the issue of 
uniform integrability, which is not entirely trivial since $w$
itself is not a priori in $L^1$ and so $W$ is not
necessarily uniformly
integrable so $W\notin A_p$.
This problem, however, disappears when one is working
on manifolds with finite measure;
in that case one can actually prove that we have 
equality between the two notions of weight, see 
\cite[Theorem 4.2]{Li2019}.
\end{remark}
The weighted Korn estimates now follow from the 
above weighted Burkholder estimate:
\begin{proof}
Define the two martingales
$\mathcal Y^{T,n}_t\mathrel{:}=U f(Y^{T,n}_t)$ and 
\begin{equation*}
(\mathcal Z^{T,n}_t)_{i,j}\mathrel{:}=
\sum_{k=1}^d\int_0^t \partial_i U f_{k,j}(Y^{T,n}_s)
\diff \mathcal B^k_s
\end{equation*}
and similarly the non-truncated versions
$\mathcal Z^T_t$ and $\mathcal Y^T_t$.
The same proof as in Theorem~\ref{thm:Korn}
allows us to obtain
\begin{equation*}
\mathcal Z^{T,n}\preccurlyeq \mathcal Y^{T,n}.
\end{equation*}
To connect the martingale-weighted Burkholder
inequality 
with our weighted estimates for $R\otimes R$,
first notice that
\begin{align*}
\|R\otimes R(f)\|^p_{L^p(w)}
&=\int_{\R^d} |R\otimes R(f)|^p w\diff x\\
&=\lim_{T\to \infty}\int_{\R^d} |
\mathbb E\left(\mathcal Z^T_T|Y^T_T=(x,0)\right)|^pw\diff x\\
&\le \lim_{T\to \infty}\int_{\R^d} \mathbb E
\left(|\mathcal Z_T^T|^p|Y_T^T=(x,0)\right)w \diff x\\
&=\lim_{T\to \infty}\int_{\R^d} 
\mathbb E\left(|\mathcal Z_T^T|^p Uw(Y_T^T)|Y_T^T=(x,0)\right)
\diff x\\
&\le \limsup_T\vvvert{\mathcal Z^T}^p_{L^p(W^T)}.
\end{align*}
Applying Theorem~\ref{thm:burkholderweighted}
to $\mathcal Z^{T,n}$ and $\mathcal Y^{T,n}$, we get
\begin{align*}
\vvvert{\mathcal Z^{T,n}}_{L^p(W^{T,n})}
&\lesssim (p^*-1) [W^{T,n}]_{A_p^{\mathrm{mart}}}
^{\max\left(\frac1{p-1
},1\right)}\vvvert{\mathcal Y^{T,n}}_{L^p(W^{T,n})}
\\&\le (p^*-1)[w]_{A_p^{\text{\upshape Heat}}}^
{\max\left(\frac1{p-1
},1\right)} \|f\|_{L^p(w)}.
\end{align*}
Combining this with the previous estimate and standard limiting
arguments (recall that $f\in C^\infty_{c}$ so $U f(\cdot, t)$
decays exponentially for each $t$) yields
\begin{align*}
\|R\otimes R(f)\|_{L^p(w)}&\le
\limsup_T\limsup_n\vvvert{\mathcal Z^{T,n}}_{L^p(W^{T,n})}\\
&\lesssim (p^*-1)\limsup_T\limsup_n
[W^{T,n}]_{A_p^{\text{\upshape mart}}}^{\max\left(\frac{1}{p-1
},1\right)} \vvvert{\mathcal Y^{T,n}}_{L^p(W^{T,n})}\\
&\le (p^*-1)
[w]_{A_p^{\text{\upshape Heat}}}^{\max\left(\frac{1}{p-1
},1\right)} \|f\|_{L^p(w)},
\end{align*}
which is enough to prove the theorem
thanks to \eqref{eq:RieszKorn1}.
\end{proof}
\section{Concluding remarks}
\subsection{Burkholder functions and
quasiconvexity}
Burkholder's functional, in its form
\[\mathcal G_p(|P_{\mathrm{CO}^-(2)}(A)|,|P_{\mathrm{CO}^+(2)}(A)|),
\]
has been the subject of very extensive study, in particular
on its conjectured quasiconvexity at $0$ (see \cite{KariAstala2024}, \cite{Astala2023TheTheory}, \cite{Guerra2022}). In view of the lower bound in Proposition 
\ref{prop:lowerconvex}, investigating which properties
generalise to the decomposition
$X\oplus X^\perp$ in $M_d(\R)$ is a natural
direction of research.
In particular, the quasiconvexity
at $0$ of the generalised Burkholder functional
seems interesting: in the general setting, many 
rather useful symmetries (such as isotropy) break, and
as $X$ gets smaller the bad part of $\mathcal G_p$
(i.e. $|P_{X^\perp}(A)|$) gets harder to control.
This means that answering the quasiconvexity question
negatively (if at all possible) should be easier 
in the generalised case, especially in high dimensions
(where the ratio $\mathrm{dim}(X)/\mathrm{dim}(M_d(\R))$
gets smaller and smaller). Other than being
useful in the study of Morrey's problem, such an 
example could also be helpful in the study
of Iwaniec's conjecture, as it would 
give an indication of
which properties one has to leverage to prove the 
conjecture.
\subsection{Extensions to other spaces}
General Riemannian manifolds admit 
an inequality of Korn-type (see~\cite{ChenJost}):
it would be interesting to see whether
the argument applied above admits an extension to such a general
setting.

In a different direction, one can ask whether estimates such as 
this also hold for domains $\Omega\subset \R^d$ when we impose 
different boundary conditions rather than 
$u\in W^{1,p}_0(\Omega)$. One popular such condition is
$\int_\Omega \mathcal A(u)=0$, i.e. the $0$-momentum condition; 
with these conditions, however, the
geometry of the domain becomes much more important
(see~\cite{JohnKorn} and~\cite{Lewicka2016OnConditions}).
For that formulation of Korn's inequality,
and similarly for Korn's inequality in its second form,
the calculus of variations approach used here
is not applicable any more
(at least, not as straightforwardly),
so such a study seems to be much more complicated.
\subsection*{Acknowledgements}
I wish to thank Professor Kristensen 
for bringing this problem to my attention,
for the many useful conversations on it
and for his very helpful comments
on previous drafts of this work.
I am also deeply grateful to
Professors P. Ivanisvili, A. Osękowski
and S. Petermichl for very kindly
answering my questions.
Finally, I would like to thank 
A. B. Bernardi for her support
during the writing of this paper.
\printbibliography
\end{document}